\documentclass[11pt]{amsart}
\pdfoutput=1

\usepackage{amssymb,latexsym}
\usepackage{amsmath,amsfonts,amsthm}
\usepackage{stmaryrd}
\usepackage{enumerate}
\usepackage{enumitem}
\usepackage[T1]{fontenc}
\usepackage[latin1]{inputenc}
\usepackage[english]{babel}

\usepackage{graphicx}
\usepackage{texdraw}

\usepackage[bookmarks=true,%
    colorlinks=true,
    linkcolor=blue,%
    citecolor=blue,%
    filecolor=blue,%
    menucolor=blue,%
    urlcolor=blue,%
    breaklinks=true]{hyperref}

\usepackage{caption}

\usepackage[all]{xy}
\CompileMatrices
\def\cxymatrix#1{\xy*[c]\xybox{\xymatrix#1}\endxy}

\theoremstyle{plain}
\newtheorem{theorem}{Theorem}[section]
\theoremstyle{definition}
\newtheorem{proposition}[theorem]{Proposition}
\newtheorem{lemma}[theorem]{Lemma}
\newtheorem{definition}[theorem]{Definition}

\newtheorem{remark}[theorem]{Remark}
\newtheorem{corollary}[theorem]{Corollary}

\numberwithin{equation}{section}

\newcommand{\Z}{\mathbb Z}
\newcommand{\Q}{\mathbb Q}
 
\newcommand{\C}{\mathbb C}
\renewcommand{\H}{\mathbb H}

\newcommand{\T}{\mathcal T}
\newcommand{\E}{\mathcal E}

\newcommand{\Matrix}[4]{\left(\begin{smallmatrix}#1&#2\\#3&#4\end{smallmatrix}\right)}
\newcommand{\Vector}[2]{\left(\begin{smallmatrix}#1\\#2\end{smallmatrix}\right)}

\renewcommand{\P}{\mathcal P}

\DeclareMathOperator{\PSL}{PSL}

\DeclareMathOperator{\SL}{SL}

\DeclareMathOperator{\Conj}{Conj}

\DeclareMathOperator{\red}{red}

\DeclareMathOperator{\TE}{T}

\usepackage{scalerel} 
\def\Mtildehat{\scalerel*{\widehat{\widetilde M}}{\widetilde{M^2}}}
\def\Ttildebar{\scalerel*{\widetilde{\overline\T}}{\widetilde{T^2}}}

\usepackage{ifpdf}
\ifpdf
\fi

\usepackage[left=2.8cm,right=2.8cm,bottom=3.1cm,top=3cm]{geometry}

\author{Matthias Goerner}
\address{Pixar Animation Studios \\
         1200 Park Avenue\\
         Emeryville, CA 94608, USA \newline
         {\tt \url{http://www.unhyperbolic.org/}}}
\email{enischte@gmail.com}

\title[Triangulation independent Ptolemy varieties]{Triangulation independent Ptolemy varieties}
\author{Christian K. Zickert}
\address{University of Maryland \\
         Department of Mathematics \\
         College Park, MD 20742-4015, USA \newline
         {\tt \url{http://www2.math.umd.edu/~zickert}}}
\email{zickert@math.umd.edu}

\thanks{C.~Z.~was supported by DMS-13-09088. M.~G.~was partially supported by DMS-11-07452. \\
\newline
1991 {\em Mathematics Classification.} Primary 57N10, 57M27, 57M50. Secondary 13P10.
}

\date{}
\begin{document}
\maketitle
\begin{abstract}
The Ptolemy variety for $\SL(2,\C)$ is an invariant of a topological ideal triangulation of a compact $3$-manifold $M$. It is closely related to Thurston's gluing equation variety. The Ptolemy variety maps naturally to the set of conjugacy classes of boundary-unipotent $\SL(2,\C)$-representations, but (like the gluing equation variety) it depends on the triangulation, and may miss several components of representations. In this paper, we define a Ptolemy variety, which is independent of the choice of triangulation, and detects all boundary-unipotent irreducible $\SL(2,\C)$-representations. We also define variants of the Ptolemy variety for $\PSL(2,\C)$-representations, and representations that are not necessarily boundary-unipotent. In particular, we obtain an algorithm to compute all irreducible $\SL(2,\C)$-characters as well as the full $A$-polynomial. All the varieties are topological invariants of $M$.
\end{abstract}
\section{Introduction}
The Ptolemy variety $P_n(\T)$ of a topological ideal triangulation $\T$ of a compact $3$-manifold $M$ was defined by Garoufalidis, Thurston and Zickert~\cite{GaroufalidisThurstonZickert}. It gives coordinates (called Ptolemy coordinates) for boundary-unipotent $\SL(n,\C)$-representations of $\pi_1(M)$ in the sense that each point in $P_n(\T)$ determines a representation (up to conjugation). The Ptolemy variety is explicitly computable for many census manifolds when $n=2$ or $3$ (see~\cite{CURVE,FalbelKoseleffRouillier} for a database), and invariants such as volume and Chern-Simons invariant can be explicitly computed from the Ptolemy coordinates. The Ptolemy variety, however, depends on the triangulation and may miss several components of representations. 

We focus exclusively on the case when $n=2$, so we omit the subscript $n$ on the Ptolemy variety. Our goal is to define a refined Ptolemy variety $\overline P(\T)$, which is a topological invariant of $M$ and is guaranteed to detect all irreducible boundary-unipotent $\SL(2,\C)$-representations. 
We also define refined variants of the Ptolemy variety for $\PSL(2,\C)$-representations~\cite{GaroufalidisThurstonZickert,PtolemyField}, and for the enhanced Ptolemy variety~\cite{ZickertEnhancedPtolemy} for $\SL(2,\C)$-representations that are not necessarily boundary-unipotent. The refined variant of the latter detects all irreducible $\SL(2,\C)$-representations. 

A representation in $\SL(2,\C)$ or $\PSL(2,\C)$ is \emph{boundary-unipotent}, respectively \emph{boundary-Borel}, if it takes each peripheral subgroup to a conjugate of the subgroup $P$, respectively $B$, where $B$ is the group of upper triangular matrices, and $P\subset B$ is the subgroup of matrices with $1$ on the diagonal.

\subsection{Motivation: Thurston's gluing equations} In his famous notes Thurston~\cite{ThurstonNotes} wrote down a system of polynomial equations (called~\emph{gluing equations}) for a compact 3-manifold $M$ with a topological ideal triangulation. The system consists of a variable (called a \emph{shape}) $z_i$ for each simplex, and an equation for each edge. The set of solutions with $z_i\in\C P^1\setminus\{0,1,\infty\}$ is called the \emph{gluing equation variety}, and each point in this variety determines (up to conjugation) a $\PSL(2,\C)$-representation of $\pi_1(M)$. Solutions with some $z_i$ being 0, 1 or $\infty$ are called \emph{degenerate}, and may not give rise to a unique representation, or to any representation at all. The gluing equation variety and the set of representations detected depend on the triangulation and may be empty.

In~\cite{Segerman} Segerman considers a generalization of Thurston's gluing equation variety, where the shapes $z_i$ are replaced by variables in the field $\C((\zeta))$ of Laurent series in a variable $\zeta$. He proves the existence of a triangulation such that his generalized variety detects all irreducible representations except those that have image in a generalized dihedral group. His approach uses normal surfaces and it does not yield explicit defining equations for the generalized variety. Our approach focuses on Ptolemy coordinates and is completely combinatorial, and detects all irreducible representations, even those with image in a generalized dihedral group. It is possible that our methods could also be used to make Segerman's generalized variety more explicit, but we shall not pursue this idea here. Our paper does not require prior knowledge of Thurston's gluing equations or Segerman's work.


\subsection{Decorated representations}
A convenient language for studying both the gluing equation variety and the Ptolemy variety is that of \emph{decorated representations} (see \cite{ZickertDuke,GaroufalidisThurstonZickert}). We give a brief overview assuming for simplicity that $M$ has a single torus boundary component (the general case and precise definitions are recalled in Section~\ref{sec:Decorations}). Let $G$ be either $\PSL(2,\C)$ or $\SL(2,\C)$. A \emph{decoration} of a representation $\rho\colon\pi_1(M)\to G$ may be thought of as a choice of a point in $\C P^1$ stabilized by $\rho(\pi_1(\partial M))$. In particular, a representation generically has two decorations\footnote{As an example, \texttt{m004(10,11)} and \texttt{m004(-10,-11)} in SnapPy give the same geometric representation but different shapes corresponding to different decorations.}, but may have only one (if $\rho$ is boundary-unipotent), infinitely many (if $\rho(\pi_1(\partial M))$ is trivial), or none at all (if $\rho$ is not boundary-Borel). Any decorated representation determines shapes satisfying Thurston's gluing equations.

Since $\C P^1=\SL(2,\C)\big/B$ we shall sometimes refer to a decoration as a \emph{B-decoration}. If $\rho$ is boundary-unipotent, a (stronger) decoration by points in $\SL(2,\C)\big/P$ exists which we refer to as a \emph{P-decoration}.

In the language of decorations, Thurston's gluing equation variety parametrizes \emph{generically $B$-decorated} boundary-Borel $\PSL(2,\C)$-representations, where a decoration is generic if the corresponding solution to Thurston's gluing equations is non-degenerate. Whether or not a representation has a generic decoration depends on the choice of triangulation.

\subsection{Ptolemy varieties in brief}\label{sec:PtolemyBrief}
The Ptolemy variety $P(\T)$ is given by a \emph{Ptolemy coordinate} for each edge and a \emph{Ptolemy relation} for each simplex, and consists of the set of solutions where each coordinate is nonzero (see Section~\ref{sec:PtDecCocReview} for a more detailed review). Just like the gluing equation variety it parametrizes decorated representations, and depends on the triangulation. More precisely, $P(\T)$ parametrizes generically $P$-decorated boundary-unipotent $\SL(2,\C)$-representations. There are three natural modifications of $P(\T)$ (see Section~\ref{sec:OtherVariants}):

\begin{itemize}
\item A Ptolemy variety $P^\sigma(\T)$, with $\sigma\in H^2(M,\partial M;\Z/2\Z)$, for boundary-unipotent $\PSL(2,\C)$-representations with obstruction class to lifting to a boundary-unipotent $\SL(2,\C)$-represen\-tation given by $\sigma$~\cite{GaroufalidisThurstonZickert,PtolemyField}.
\item An ``enhanced'' Ptolemy variety $\E P(\T)$ for $\SL(2,\C)$-representations that are not necessarily boundary-unipotent~\cite{ZickertEnhancedPtolemy}. 
\item A ``reduced'' Ptolemy variety $P(\T)_{\red}$ for $B$-decorations instead of $P$-decorations. It is the quotient of $P(\T)$ by a $(\C^*)^c$-action, where $c$ is the number of boundary components of $M$~\cite{GaroufalidisThurstonZickert,PtolemyField}.\end{itemize}

In this paper we construct another modification:

\begin{itemize}
\item A triangulation independent refinement $\overline P(\T)$.
\end{itemize}
The modifications can all be applied independently, e.g., one also has $P(\T)_{\red}^\sigma$ and $\E \overline P(\T)$, but we shall not consider enhanced Ptolemy varieties with obstruction classes here (see Remark~\ref{rm:EnhancedPSLExample} for a brief discussion).

\begin{remark}There is an explicit regular map from each (triangulation dependent) Ptolemy variety to the gluing equation variety, but we shall not need this here (see e.g.~\cite{GaroufalidisThurstonZickert,GaroufalidisGoernerZickert,ZickertEnhancedPtolemy}).
\end{remark}

\subsection{Statement of results} 
We briefly state our main results about the triangulation independent Ptolemy varieties. A summary of their definition is given in Section~\ref{sec:Algorithm}, and explicit computations are given in Section~\ref{sec:Examples}. Let $M$ be a compact, orientable $3$-manifold with non-empty boundary, and let $\T$ be a topological ideal triangulation of $M$.

Any $P$-decorated boundary-unipotent representation $\rho\colon\pi_1(M)\to\SL(2,\C)$ assigns Ptolemy coordinates to the edges of $\T$ (see Section~\ref{sec:PtDecCocReview}). If all Ptolemy coordinates are zero, we say that the decorated representation is \emph{totally degenerate}. This notion turns out to be independent of the (ideal) triangulation and implies that the representation is reducible (see Section~\ref{sec:TotallyDegenerate}).

Recall that although the set of representations of $\pi_1(M)$ is a variety, the set of representations up to conjugation is not a variety in general. The same holds for decorated representations. However, if we discard the totally degenerate decorations, we get a variety:

\begin{theorem}\label{thm:MainResult}
The conjugation action on the set of non totally degenerate decorated boundary-unipotent representations has a geometric quotient. For each triangulation $\T$ of $M$, the invariant Ptolemy variety $\overline P(\T)$ provides explicit coordinates, i.e.~we have an isomorphism of varieties
\begin{equation}\label{eq:MainResult}
\overline P(\T)\cong\left\{\txt{Not totally degenerate, $P$-decorated, \\boundary-unipotent $\pi_1(M)\to\SL(2,\C)$}\right\}\big/\Conj.
\end{equation} 
In particular, $\overline P(\T)$ is independent of the triangulation. Morever, $P(\T)$ embeds in $\overline P(\T)$ as a Zariski open subset. \qed
\end{theorem}

\begin{definition}A representation is \emph{boundary-nondegenerate} if its restriction to $\pi_1(\partial_i M)$ is non-trivial for each boundary component $\partial_i M$ of $M$. 
\end{definition}

Theorems~\ref{thm:OneToOneOverNonDegenerate}, \ref{thm:MainThmPSL} and \ref{thm:MainThmEnhanced} below are elementary corollaries of Theorem~\ref{thm:MainResult} and its analogues for the modified variants. They were previously known for representations admitting a \emph{generic} decoration (see~\cite{PtolemyField,GaroufalidisThurstonZickert,ZickertEnhancedPtolemy}). 

\begin{theorem}\label{thm:OneToOneOverNonDegenerate}
The map
\begin{equation}\label{eq:RedOneToOne}
\overline P(\T)_{\red}\to\left\{\txt{Boundary-unipotent\\$\pi_1(M)\to\SL(2,\C)$}\right\}\big/\Conj
\end{equation}
induced from~\eqref{eq:MainResult} by ignoring the decoration has image containing all irreducible representations and is one-to-one over the set of irreducible boundary-nondegenerate representations.\qed
\end{theorem}

\begin{theorem}\label{thm:MainThmPSL}
Let $k=\big\vert H^1(M,\partial M;\Z/2\Z)\big\vert$. The map
\begin{equation}\label{eq:RedkToOne}
\overline P^\sigma(\T)_{\red}\to\left\{\txt{Boundary-unipotent\\$\pi_1(M)\to\PSL(2,\C)$\\with obstruction class $\sigma$}\right\}\big/\Conj
\end{equation}
has image containing all irreducible representations and is $k:1$ over the set of irreducible, boundary-nondegenerate representations. \qed
\end{theorem}

\begin{theorem}\label{thm:MainThmEnhanced}
Suppose $M$ has $c$ boundary components, all of which are tori. There is a map
\begin{equation}\label{eq:RedTwoToOne}
\E \overline P(\T)_{\red}\to\left\{\txt{Boundary-Borel\\$\pi_1(M)\to\SL(2,\C)$}\right\}\big/\Conj
\end{equation}
with image containing all irreducible representations. It is generically $2^c:1$ over the irreducible, boundary-nondegenerate representations. 
Moreover, the projection to the $(m_s,l_s)$ coordinates is the variety of eigenvalues of $\mu_s$ and $\lambda_s$.\qed
\end{theorem}

This can be used to compute the $A$-polynomial; see Section~\ref{sec:Apolyexample} for an example.

\begin{remark}
The image of~\eqref{eq:RedOneToOne}, \eqref{eq:RedkToOne}, and \eqref{eq:RedTwoToOne} may contain reducible representations as well (see e.g.~the example in Section~\ref{sub:Reducible}), but such are necessarily boundary-degenerate (see Proposition~\ref{prop:TotallyDegenerate}). The preimage over a representation which is not boundary-nondegenerate is typically (although not always; see Remark~\ref{rm:Freedom}) higher dimensional.
\end{remark}

\begin{remark}
One could detect all representations (including all reducible ones) by performing a single barycentric subdivision~\cite{GaroufalidisThurstonZickert}. The resulting variety, however, is not a topological invariant since each non-ideal vertex (generically) increases the dimension by one. The growth in the number of simplices and the expected dimension also makes computations infeasible. More importantly, invariants such as the loop invariant~\cite{DG}, important in quantum topology, are only defined when all vertices are ideal.
\end{remark}

\begin{remark}
Our varieties are really schemes over $\Z$. They are typically not irreducible, and the coordinate rings may have nilpotents.
\end{remark}

\subsection{Definition and computation of the refined Ptolemy varieties - an overview}\label{sec:Algorithm}
We consider only $\overline P(\T)$ here. The modified variants are defined and computed in a similar way.
As mentioned earlier, a decorated representation defines Ptolemy coordinates, and thus gives rise to a distinguished set of edges, namely those having Ptolemy coordinate zero. The idea is to consider all possible such sets. It turns out that they have to satisfy a transitivity property.

\begin{definition} A \emph{transitive edge set} is a set $E$ of edges of $\T$ satisfying that if any two edges of a face are in $E$, so is the third. The edges in a transitive edge set are called \emph{zero-edges}. We sometimes denote $E$ by $(\T,E)$ to stress the significance of the triangulation.
\end{definition}

There are always at least two transitive edge sets: the \emph{non-degenerate} edge set with no zero-edges, and the \emph{totally degenerate} edge set where all edges are zero-edges. In Section~\ref{sec:InvariantPtolemy} we define for all transitive edge sets $E$ (except the totally degenerate one) a Ptolemy variety $P(\T,E)$, parametrizing decorated representations whose Ptolemy coordinates are zero exactly at the edges in $E$.
The triangulation independent Ptolemy variety $\overline P(\T)$ is then defined as the union of all the $P(\T,E)$ (see Definition~\ref{def:InvariantPtolemy}). For the non-degenerate edge set $P(\T,E)$ is simply $P(\T)$. If each face has at most one zero-edge (the \emph{edge-degenerate} case; see Definition~\ref{def:PartitionType}), $P(\T,E)$ is defined by the usual Ptolemy relations together with an \emph{edge relation}~\eqref{eq:EdgeRelations} for each zero-edge.
If one or more faces consists entirely of zero-edges, one can reduce to the edge-degenerate case by considering an auxilary triangulation $\T'$ obtained from $\T$ by performing an explicit sequence of 2-3 moves.

\subsubsection{Algorithm}
We describe an algorithm that is guaranteed to find all boundary-unipotent irreducible representations of a manifold, no matter what triangulation $\T_0$ is used as input. The steps involve triangulations obtained from $\T_0$ by performing 2-3 moves, but these are purely auxiliary.

A simplex is \emph{degenerate} if all six edges are zero. A face is \emph{degenerate} if all 3 edges are zero.

\begin{itemize}[leftmargin=0.8in, rightmargin=0.2in]
\item[Step 0:] Compute all transitive edge sets $(\T_0,E_0)$ ignoring the totally degenerate one.
\item[Step 1:] If $(\T,E)$ has a degenerate simplex, replace it by a descendant (Definition~\ref{def:Descendant}) $(\T',E')$ after performing a 2-3 move at a face adjacent to a non-degenerate and a degenerate simplex. This procedure reduces the number of degenerate simplices by one. Repeat until all simplices are non-degenerate.
\item[Step 2:] If $(\T,E)$ has a degenerate face, replace it by the set of descendants $(\T',E')$ where $\T'$ is the triangulation obtained from $\T$ by performing a 2-3 move on each degenerate face. No descendants have degenerate faces.
\item[Step 3:] For each $(\T,E)$ from step 2 compute the primary decomposition of the reduced Ptolemy variety $P(\T,E)_{\red}$.
\item[Step 4:]  For each zero dimensional component of $P(\T,E)_{\red}$, compute the corresponding boundary-unipotent $\SL(2,\C)$-representations via the Bruhat cocycle (Section~\ref{sec:BruhatCocycle}). 
The output is a matrix of exact algebraic expressions for each generator of $\pi_1(M)$ in the face pairing presentation.
\end{itemize}

\begin{remark} If a Ptolemy variety contains a higher dimensional component $C$, one can compute the tautological representation $\pi_1(M)\to \SL(2,F(C))$, where $F(C)$ is the function field of $C$. See Section~\ref{sub:TautologicalRep} for an example.
\end{remark}

\begin{remark} Among the 61911 manifolds in the SnapPy census \texttt{OrientableCuspedCensus} there is no manifold with more than 53 transitive edge sets (average 16). Most of these edge sets are only edge-degenerate (72\%), so steps 1 and 2 above are rarely needed.
\end{remark}

\begin{remark} In case there are only two transitive edge sets it follows immediately from the definition that $\overline P(\T)=P(\T)$, so in this case the standard Ptolemy variety is guaranteed to detect all boundary-unipotent representations. Unfortunately, the only manifolds in the SnapPy census \texttt{OrientableCuspedCensus} with only two transitive edge sets are: \texttt{m003}, \texttt{m004}, \texttt{m015}, \texttt{m016}, \texttt{m017}, \texttt{m019}, \texttt{m118}, \texttt{m119}, \texttt{m180}, \texttt{m185}.
\end{remark}


\section{Ptolemy coordinates, decorations and cocycles}\label{sec:PtDecCocReview}
This section summarizes results of \cite{GaroufalidisThurstonZickert} and~\cite{GaroufalidisGoernerZickert}; see also~\cite{PtolemyField} for a review of the case when $n=2$ with many worked out examples.

Let $M$ be a compact, orientable $3$-manifold with non-empty boundary and let $\widetilde M$ be the universal cover of $M$. Let $\widehat M$ and $\Mtildehat$ denote the spaces obtained from $M$ and $\widetilde M$, respectively, by collapsing each boundary component to a point. Given a topologically ideal triangulation $\T$ of $M$ we refer to the cells as \emph{vertices}, \emph{edges}, \emph{faces} and \emph{simplices} of $\T$. 

\subsection{Ptolemy assignments}

Fix a (topological ideal) triangulation $\T$ of $M$. An \emph{ordered simplex} is a simplex together with an ordering of its vertices.
\begin{definition}
A \emph{Ptolemy assignment} on an ordered simplex $\Delta$ is an assignment of a nonzero complex number $c_{ij}$ to each oriented edge $\varepsilon_{ij}$ of $\Delta$ satisfying the \emph{Ptolemy relation}
\begin{equation}\label{eq:PtolemyRelation}
c_{03}c_{12}+c_{01}c_{23}=c_{02}c_{13},
\end{equation}
and the \emph{edge orientation relations} $c_{ji}=-c_{ij}$.
\end{definition}

\begin{definition}
A \emph{Ptolemy assignment} on $\T$ is a Ptolemy assignment on each simplex $\Delta_k$ such that the Ptolemy coordinates $c_{ij,k}$ satisfy the \emph{identification relations}
\begin{equation}
c_{ij,k}=c_{i'j',k'},\qquad\text{when}\quad \varepsilon_{ij,k}\sim\varepsilon_{i'j',k'} 
\end{equation}
where $\sim$ denotes identification of oriented edges.
\end{definition}

\begin{definition}\label{def:PtolemyVariety} The~\emph{Ptolemy variety} $P(\T)$ is the Zariski open subset where $c_{ij,k}\neq 0$ of the zero set of the ideal in $\Q[\{c_{ij,k}\}]$ generated by the Ptolemy relations, the identification relations, and the edge orientation relations. As a set it is equal to the set of Ptolemy assignments on $\T$.
\end{definition}

\begin{remark}  The purpose of the edge orientation relations is to make the Ptolemy variety ``order agnostic'', i.e.~independent of the choice of vertex orderings. Whenever convenient we shall eliminate the edge orientation relations and only consider $c_{ij,k}$ with $i<j$. By choosing once and for all representatives for each (oriented) edge of $\T$, the Ptolemy variety can be given by a variable for each edge and a relation for each simplex.
\end{remark}

\subsection{Decorations}\label{sec:Decorations}
Let $B\subset\SL(2,\C)$ denote the subgroup of upper triangular matrices, and $P$ the subgroup of upper triangular matrices with $1$'s on the diagonal. 
Let $G=\SL(2,\C)$ and let $H\subset G$ denote either $P$ or $B$.

\begin{definition}
A $(G,H)$-representation is a representation $\pi_1(M)\to G$ taking each peripheral subgroup $\pi_1(\partial_i M)$ to a conjugate of $H$. Such are called \emph{boundary-unipotent} for $H=P$ and \emph{boundary-Borel} for $H=B$. 
\end{definition}

\begin{definition} A $(G,H)$-representation is \emph{boundary-nondegenerate} if its restriction to $\pi_1(\partial M_i)$ is
non-trivial for each boundary component $\partial_i M$ of $M$.
\end{definition}

Let $I(\widetilde M)$ denote the set of ideal vertices of $\widetilde M$, i.e.~vertices of $\Mtildehat$ with the triangulation induced by $\T$. Note that each ideal point corresponds to a boundary component of $\widetilde M$, so $I(\widetilde M)$ is independent of $\T$.

\begin{definition}
A \emph{decoration} of a simplex $\Delta$ is an assignment of a coset $g_iH$ to each vertex $v_i$ of $\Delta$. A decoration is thus a tuple $(g_0H,g_1H,g_2H,g_3H)$, and we consider two decorations to be equal if the tuples differ by multiplication by an element in $\SL(2,\C)$. A decoration is \emph{generic} if the cosets are distinct as $B$-cosets (distinct if $H=B$).
\end{definition}

\begin{definition}\label{def:Decoration}
Let $\rho$ be a $(G,H)$-representation. A \emph{decoration} of $\rho$ is a $\rho$-equivariant map 
\begin{equation}\label{eq:DevelopingMap}
D\colon I(\widetilde M)\to G/H,
\end{equation}
i.e.~an equivariant assignment of $H$-cosets to the ideal points of $\widetilde M$. When the representation plays no role, we shall refer to a decorated representation simply as a decoration. Since $gD$ is a decoration of $g\rho g^{-1}$ if $D$ is a decoration of $\rho$, we consider two decorations to be equal if they differ by left multiplication by an element in $G$.  
\end{definition}

\begin{remark} A boundary-unipotent representation is also boundary-Borel. We shall refer to decorations as $P$-decorations or $B$-decorations depending on context. 
\end{remark}

\begin{definition}
A decoration is \emph{generic} if the induced decoration of each simplex of $\T$ is generic. 
\end{definition}

\begin{remark} Every $(G,H)$-representation has a decoration, but whether a decoration is generic depends on $\T$.
\end{remark}

\begin{remark} A $B$-coset determines a point in $\partial\overline\H^3=\C\cup\{\infty\}$, the boundary of hyperbolic $3$-space, via $gB\leftrightarrow g\infty$. A decoration of a boundary-Borel representation thus determines a developing map (see~e.g.~Zickert~\cite{ZickertDuke}) assigning ideal simplex shapes to the simplices of $\T$. A decoration is generic if and only if all shapes are non-degenerate. We shall not need this here.
\end{remark}

\subsubsection{Freedom in the choice of decoration}\label{sec:Freedom}
If we choose points $e_i\in I(\widetilde M)$, one for each boundary component $\partial_i M$ of $M$, a decoration is uniquely determined by the cosets $D(e_i)$. The freedom in the choice of $D(e_i)$ is determined by the image of the boundary components.
\begin{proposition}\label{prop:Freedom} 
Let $\rho$ be a $(G,B)$-representation and let $e_i$ be as above.
\begin{enumerate}[label=(\roman*)]
\item If $\rho(\pi_1(\partial_i M))$ is trivial, $D(e_i)$ may be chosen arbitrarily.
\item If $\rho(\pi_1(\partial_i M))$ is non-trivial and unipotent, $D(e_i)$ is uniquely determined by $\rho$.
\item If $\rho(\pi_1(\partial_i M))$ is non-trivial and diagonalizable, $D(e_i)$ is determined by $\rho$ up to a $\Z/2\Z$-action.\qed
\end{enumerate} 
\end{proposition}

\begin{corollary}\label{cor:Freedom}
A boundary-nondegenerate $(G,P)$-representation has a unique $B$-decoration. A boundary-nondegenerate $(G,B)$-representation generically has $2^c$ decorations, where $c$ is the number of boundary components.
\end{corollary}

\begin{remark}\label{rm:Freedom}
Different choices of $D(e_i)$ may give rise to equal decorations, i.e.~decorations differing only by left multiplication by an element in $G$. Hence, even when a boundary-component is collapsed, there may be only finitely many decorations (see e.g.~the example in Section~\ref{sub:Reducible}).
\end{remark}

\subsection{Natural cocycles}\label{sec:NaturalCocycle} The triangulation $\T$ of $M$ induces a decomposition $\overline \T$ of $M$ by truncated simplices. 

\begin{definition}\label{def:NaturalCocycle}
A \emph{natural cocycle} on a truncated simplex $\overline\Delta$ is a labeling of the oriented edges by elements in $\SL(2,\C)$ such that
\begin{enumerate}[label=(\roman*)]
\item\label{itm:1} The product around each face (triangular and hexagonal) is $I$, the identity matrix.
\item\label{itm:3} Short edges are labeled by elements $\beta_{ij}^k$ in $P$.
\item\label{itm:4} Long edges are labeled by counter diagonal elements $\alpha_{ij}$.
\item\label{itm:2} Flipping the orientation replaces a labeling by its inverse, i.e.~$\alpha_{ij}\alpha_{ji}=I=\beta_{ij}^k\beta_{ji}^k$.
\end{enumerate}
The indexing is such that $\alpha_{ij}$ is the labeling of the long edge from vertex $i$ to $j$, and $\beta_{ij}^k$ is the labeling of the short edge near vertex $k$ parallel to the edge from $i$ to $j$; see Figure~\ref{fig:FundamentalCorrespondence}.
\end{definition}
\begin{definition}
A natural cocycle on $\overline \T$ is a natural cocycle on each truncated simplex such that the labelings of identified edges agree.
\end{definition}

\begin{remark} Decorations, $(G,H)$-representations, and natural cocycles are also defined for $G=\PSL(2,\C)$. The analogue of Corollary~\ref{cor:Freedom} holds as well.
\end{remark}

\subsection{The diagonal action}\label{sec:DiagonalAction}
If $c$ is the number of boundary components and $T\subset G$ is the subgroup of diagonal matrices, the torus $T^c$ acts on $P$-decorations, Ptolemy assignments and natural cocycles. This action is called the~\emph{diagonal action} and is illustrated in Figure~\ref{fig:DiagonalAction}. 
\begin{definition} The quotient of $P(\T)$ by the diagonal action is called the \emph{reduced Ptolemy variety} $P(\T)_{\red}$. 
\end{definition}

\begin{figure}[htb]
\scalebox{0.68}{\input{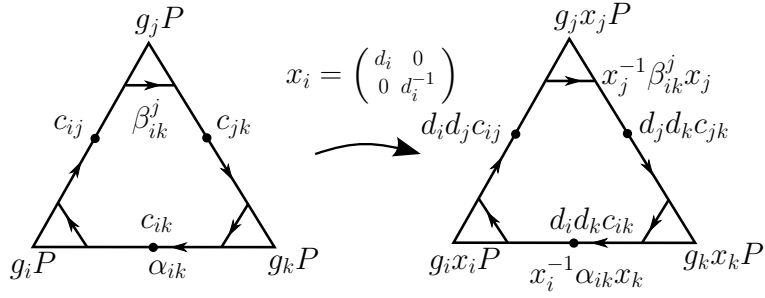}}
\caption{The diagonal action on decorations, Ptolemy assignments and natural cocycles.}\label{fig:DiagonalAction}
\end{figure}


\subsection{The fundamental correspondence}
The following result is proved in~\cite{GaroufalidisThurstonZickert} for $\SL(n,\C)$. For $n=2$ the correspondence is particularly simple, and is illustrated in Figure~\ref{fig:FundamentalCorrespondence}.
\begin{theorem} We have a one-to-one-correspondence
\begin{equation}\label{eq:FundamentalCorrespondence}
\left\{\text{Generic $P$-decorations}\right\}\overset{\text{1:1}}{\longleftrightarrow}P(\T)\overset{\text{1:1}}{\longleftrightarrow}\left\{\text{Natural cocycles on $\overline\T$}\right\}
\end{equation}
respecting the diagonal action.\qed
\end{theorem}




\begin{figure}[htb]
\scalebox{0.68}{\input{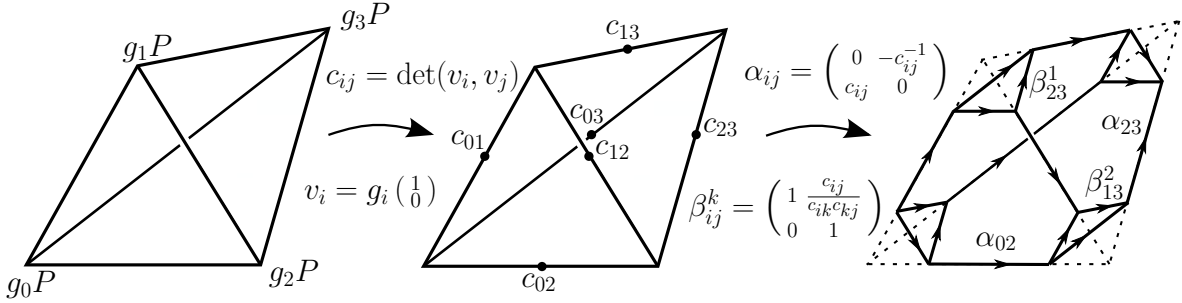}}
\caption{The fundamental correspondence.}\label{fig:FundamentalCorrespondence}
\end{figure}

\begin{remark}\label{rm:NonGenericDecoration}
A decoration $(g_0P,\dots,g_3P)$ is generic if and only if $c_{ij}\neq0$, where $c_{ij}$ are defined as in Figure~\ref{fig:FundamentalCorrespondence}. Even if a decoration is not generic, the Ptolemy relation is still satisfied.
\end{remark}



\section{Other variants of the Ptolemy variety}\label{sec:OtherVariants}
Section~\ref{sec:PSLPtolemy} summarizes results in~\cite{GaroufalidisThurstonZickert} (see also~\cite{PtolemyField}). Section~\ref{sec:Enhanced} summarizes results in \cite{ZickertEnhancedPtolemy}. 

\subsection{Ptolemy coordinates for $\PSL(2,\C)$-representations}\label{sec:PSLPtolemy}

A manifold may have boundary-unipotent representations in $\PSL(2,\C)$ that do not lift to boundary-unipotent representations in $\SL(2,\C)$.
For example, a geometric representation of a complete one-cusped hyperbolic manifold has no boundary-unipotent lifts (any lift of a longitude has trace $-2$, not $2$~\cite{Calegari}). The obstruction to existence of a boundary-unipotent lift is a class $\sigma\in H^2(\widehat M;\Z/2\Z)$. For each such class there is a Ptolemy variety $P^\sigma(\T)$.

Let $\sigma\in C^2(\widehat M;\Z/2\Z)$ be a cellular cocycle representing a class in $H^2(\widehat M;\Z/2\Z)$. Although $\sigma$ may not be a coboundary, its restriction $\sigma_k$ to each simplex $\Delta_k$ is a coboundary. Fix $\eta_k\in C^1(\Delta_k;\Z/2\Z)$ such that $\delta(\eta_k)=\sigma_k$. We identify $\Z/2\Z$ with $\{\pm 1\}$.
\begin{definition}
The Ptolemy variety $P^\sigma(\T)$ is the variety generated by the Ptolemy relations, the edge orientation relations, and the \emph{modified edge identification relations}
\begin{equation}\label{eq:EdgeIdentificationsSigma}
c_{ij,k}=(\eta_{ij,k}\eta_{i'j',k'})c_{i'j',k'},\qquad \text{ if } \varepsilon_{ij,k}=\varepsilon_{i'j',k'}.
\end{equation}
Here $\eta_{ij,k}$ denotes the value of $\eta_k$ on the $ij$-edge of $\Delta_k$.
\end{definition}
Up to canonical isomorphism this is independent of the choices of $\eta_k$, and only depends on the class of $\sigma$ in $H^2(\widehat M;\Z/2\Z)$.
\begin{remark}
In~\cite{PtolemyField}, the Ptolemy relations are modified, but not the identification relations. One easily checks that the two definitions agreee. 
\end{remark}
The analogue of the fundamental correspondence~\eqref{eq:FundamentalCorrespondence} is
\begin{equation}\label{eq:FundamentalCorrespondencePSL}
\cxymatrix{{
\left\{\txt{Generically decorated\\
$(\PSL(2,\C),P)$-representations\\with obstruction class $\sigma$}\right\}\ar@/_1pc/@{<->}[rr]_-{1:1}&P^\sigma(\T)\ar[l]_-{1:z}\ar[r]^-{z:1}&\left\{\txt{Natural $\PSL(2,\C)$-cocycles on $\overline\T$\\with obstruction class $\sigma$}\right\},}}
\end{equation}
where $z$ is the order of $Z^1(\widehat M;\Z/2\Z)$, the group of $\Z/2\Z$-valued $1$-cocycles. The maps are defined as in~\eqref{eq:FundamentalCorrespondence}, i.e.,~given by Figure~\ref{fig:FundamentalCorrespondence}; see~\cite[(9.31)]{GaroufalidisThurstonZickert} and~\cite[(1-4)]{PtolemyField}.

\begin{remark}
A coboundary in $Z^1(\widehat M;\Z/2\Z)$ acts trivially on the reduced Ptolemy variety, so the map from $P^\sigma(\T)_{\red}$ to the set of $B$-decorations is $k:1$, where $k=\big\vert H^1(\widehat M;\Z/2\Z)\big\vert$.
\end{remark}

\subsection{Non-boundary-unipotent representations}\label{sec:Enhanced}
Assume that each boundary component $\partial_s M$ is a torus, and that we have fixed a meridian $\mu_s$ and a longitude $\lambda_s$ in $H_1(\partial_s M)$. The definition of $\E P(\T)$ involves a choice of \emph{fundamental rectangle} $R_s$ for each boundary component $\partial_s M$, such that the triangulation induced on the torus obtained by identifying the sides of $R_s$ agrees with the triangulation of $\partial_s M$ induced by $\T$. It is given by the usual variables $c_{ij,k}$ as well as additional variables $m_s,l_s$ indexed by the boundary components. 
\begin{definition}\label{def:EnhancedPtolemy}
The \emph{enhanced Ptolemy variety} $\E P(\T)$ is the Zarisky open subset, where all $c_{ij,k}$ and all $m_s,l_s$ are nonzero, of the zero set of the ideal in $\Q[{\{c_{ij,k}\}}\cup\{m_s,l_s\}]$ generated by the Ptolemy relations, the edge orientation relations together with \emph{modified identification relations} of the form $c_{ij,k}=pc_{i'j',k'}$, where $p$ is a monomial in the $m_s$ and $l_s$.
\end{definition}
We refer to~\cite{ZickertEnhancedPtolemy} for the precise definition of the modified identification relations. They are illustrated in Figure~\ref{fig:FundRectIdentifications} in the case where there is a single boundary torus, and where the sides $\mu'$ and $\lambda'$ of the fundamental rectangle agree with $\mu$ and $\lambda$ (if they don't agree perform the appropriate coordinate change). Here $i_j$ denotes the triangle near vertex $i$ of simplex $j$.

\begin{figure}[htb]
\hspace{-1cm}
\begin{minipage}[b]{0.55\textwidth}
\scalebox{0.75}{\input{figures_gen/FundRectIdentifications.tex}}
\end{minipage}
\hspace{1cm}
\begin{minipage}[b]{0.25\textwidth}
\scalebox{0.92}{\input{figures_gen/EPCocycle.tex}}
\end{minipage}\\
\begin{minipage}[t]{0.55\textwidth}
\caption{The identification relations for a face pairing $\alpha$.}\label{fig:FundRectIdentifications}
\end{minipage}
\begin{minipage}[t]{0.42\textwidth}
\caption{Face pairing edges and the fundamental correspondence.}\label{fig:EPCocycle}
\end{minipage}
\end{figure}

A representation which is not boundary-unipotent does not have a $P$-decoration as defined in Section~\ref{sec:Decorations}. We shall therefore need a slightly modified definition. The definition below appears to depend on the triangulation, but one easily checks that it only depends on a choice of base point in each boundary component.
\begin{definition} Let $\rho\colon\pi_1(M)\to\SL(2,\C)$ be boundary-Borel. A \emph{$P$-decoration} of $\rho$ is a $\rho$-equivariant assignment of $P$-cosets to each triangular face of $\widetilde{\overline\T}$.
\end{definition} 

In order to state the analogue of the fundamental correspondence~\eqref{eq:FundamentalCorrespondence} we need to consider another decomposition of $M$ called the \emph{fattened decomposition}. This decomposition is obtained from the decomposition by truncated simplices by replacing each hexagonal face by a hexagonal prism. Its edges consist of the usual long and short edges together with \emph{face pairing edges} joining two corresponding vertices on each side of a prism (see Figure~\ref{fig:EPCocycle}). We refer to~\cite[Sec.~4.3.1]{ZickertEnhancedPtolemy} for details. 
\begin{definition} A \emph{fattened natural cocycle} is a cocycle on the fattened decomposition restricting to a natural cocycle on each truncated simplex and labeling face pairing edges by diagonal elements.
\end{definition}

The analogue of the fundamental correspondence is
\begin{equation}\label{eq:FundamentalCorrespondenceGeneral}
\left\{\text{Generic $P$-decorations}\right\}\overset{\text{1:1}}{\longleftrightarrow}\E P(\T)\overset{\text{1:1}}{\longleftrightarrow}\left\{\text{Fattened natural cocycles on $\overline\T$}\right\}.
\end{equation}
The maps are illustrated in Figure~\ref{fig:EPCocycle} and described in detail in~\cite[Sec.~4.3]{ZickertEnhancedPtolemy}.
\begin{remark} The projection of $\E P(\T)$ onto the $m_s,l_s$ coordinates is the variety of eigenvalues of $\mu_s$ and $\lambda_s$ under $\SL(2,\C)$-representations. In particular, if $M$ has a single torus boundary component, the one-dimensional components give rise to factors of the $A$-polynomial.
\end{remark}


\section{Generalization of the fundamental correspondence}
The fundamental correspondence between \emph{generic} decorations, Ptolemy assignments, and natural cocycles plays an important role. The Ptolemy variety gives explicit coordinates enabling concrete computations, the decorations establish the link to group homology allowing for explicit computation of the Cheeger-Chern-Simons class~\cite{GaroufalidisThurstonZickert}, and the natural cocycles allow us to explicitly recover a representation from its coordinates.

As mentioned in Remark~\ref{rm:NonGenericDecoration} a non-generic decoration still determines a Ptolemy assignment (where some coordinates are allowed to be zero) via the map in Figure~\ref{fig:FundamentalCorrespondence}, but this map is neither injective nor surjective. For example, if three Ptolemy coordinates on a face are zero, the Ptolemy relation becomes $0=0$, and we can not recover the decoration. Also, a Ptolemy assignment where all but one of the Ptolemy coordinates are 0, can never arise from a decoration. 
In this section we generalize the one-to-one correspondence between decorations and natural cocycles. The definition of the invariant Ptolemy variety is carried out in Sections~\ref{sec:InvariantPtolemy} and~\ref{sec:Invariant}.

\subsection{Cocycles, and the action by coboundaries}
Let $G=\SL(2,\C)$. Recall that the triangulation $\T$ induces a decomposition $\overline\T$ of $M$ by truncated simplices. 

\begin{definition}
A \emph{$G$-cocycle} on $M$ is a labeling of the oriented edges of $M$ satisfying \ref{itm:1} and \ref{itm:2} of Definition~\ref{def:NaturalCocycle}. A $G$-cocycle satisfying \ref{itm:3} as well is called a $(G,P)$-cocycle.
\end{definition}
A $G$-cocycle $\tau$ determines (up to conjugation) a representation $\pi_1(M)\to G$ by taking products along edge paths. Note that if $\tau$ is a $(G,P)$-cocycle this representation is boundary-unipotent and is canonically decorated. Hence, a $(G,P)$-cocycle determines a decoration.
\begin{definition}
If $D$ is a decoration and $\tau$ a $(G,P)$-cocycle, we say that $\tau$ is compatible with $D$ if the decoration determined by $\tau$ equals $D$.
\end{definition}

Let $V(\overline\T)$ denote the set of vertices of $\overline\T$. Given an oriented edge $e$ of $\overline\T$, let $e_0,e_1\in V(\overline\T)$, denote the starting and ending vertex of $e$, respectively.

\begin{definition} A \emph{zero-cochain} is a map $\eta\colon V(\overline\T)\to G$. The \emph{coboundary} of a zero-cochain $\eta$ is the $G$-cocycle labeling an oriented edge $e$ by $\eta(e_0)^{-1}\eta(e_1)$.
\end{definition}


\begin{definition} The \emph{coboundary action} of a \emph{$P$-valued} zero-cochain $\eta\colon V(\overline\T)\to P$ on a $(G,P)$-cocycle $\tau$ replaces $\tau$ by $\eta\tau$, the cocycle defined by
\begin{equation}
\eta\tau(e)=\eta(e_0)^{-1}\tau(e)\eta(e_1).
\end{equation}
\end{definition}
Note that the coboundary action does not change the decorated representation.

\subsection{Bruhat cocycles}\label{sec:BruhatCocycle}
\begin{definition} A \emph{Bruhat cocycle} is a natural cocycle as in Definition~\ref{def:NaturalCocycle}, but where long edges may be either diagonal or counter-diagonal. A zero-cochain $\eta\colon V(\overline\T)\to P$ \emph{preserves} a Bruhat cocycle $\tau$ if $\eta\tau$ is again a Bruhat cocycle.
\end{definition}
The motivation for our definition is the following corollary of the Bruhat decomposition theorem.
\begin{lemma}\label{lemma:Bruhat} Let $gP$ and $hP$ be cosets. If $gB\neq hB$ there are unique coset representatives $gx_0$ and $hx_1$ such that $(gx_0)^{-1}hx_1$ is counterdiagonal. If $gB=hB$ there are (not unique) coset representatives such that $(gx_0)^{-1}hx_1$ is diagonal.
\end{lemma}
\begin{proof}
By Bruhat decomposition, every $g\in G$ can be decomposed as $b_1wb_2$, where $w$ is either $\Matrix{0}{-1}{1}{0}$ or $\Matrix{1}{0}{0}{1}$ and $b_i\in B$. The result is an elementary consequence.
\end{proof}

\begin{corollary}\label{cor:LocalOneToOne} For every decoration $D=(g_0P,g_1P,g_2P,g_3P)$ of a simplex $\Delta$ there exists a Bruhat cocycle $\tau$ on the corresponding truncated simplex $\overline\Delta$ compatible with $D$. Moreover $\tau$ is unique up to coboundaries preserving $\tau$.
\end{corollary}
\begin{proof}
Lemma~\ref{lemma:Bruhat} provides the existence of a map $\eta\colon V(\overline\Delta)\to G$ such that $\delta\eta$ is a Bruhat cocycle (see Figure~\ref{fig:NaturalCocycleDefn}). Uniqueness up to coboundaries and compatiblity with $D$ is immediate from the construction.
\end{proof}

The following is the global analogue of Corollary~\ref{cor:LocalOneToOne}.
\begin{theorem}\label{thm:OneToOneGeneral} There is a one-to-one correspondence 
\begin{equation}\label{eq:OneToOneGeneral}
\left\{\txt{Decorated boundary-unipotent\\
$\SL(2,\C)$-representations}\right\}\overset{1:1}{\longleftrightarrow}\left\{\txt{Bruhat cocycles $\tau$ up to\\coboundaries preserving $\tau$}\right\}.
\end{equation}
\end{theorem}
\begin{proof} We must show that given a decoration $D$ there exists a Bruhat cocycle $\tau$ compatible with $D$ and that $\tau$ is unique up to the action by coboundaries preserving $\tau$. A decoration $D\colon I(\widetilde M)\to\SL(2,\C)/P$ of a representation $\rho$ determines a map $\Gamma\colon V(\Ttildebar)\to \SL(2,\C)/P$ defined by taking a vertex near an ideal vertex $v$ to the coset assigned to $v$. We shall construct a $\rho$-equivariant lift of $\Gamma$ to a map $\eta\colon V(\Ttildebar)\to \SL(2,\C)$ such that $\tilde\tau=\delta\eta$ is a natural cocycle on $\Ttildebar$ descending to a natural cocycle $\tau$ on $\overline\T$. Any cocycle compatible with $D$ arises from this construction (since $\widetilde M$ is simply connected all lifts are coboundaries). Fix an orientation of each long edge $e$ of $\overline\T$, and a lift $\tilde e$ of $e$. Each vertex of $\Ttildebar$ is then in the $\pi_1$ orbit of exactly one of the vertices $\tilde e_0$ and $\tilde e_1$. By $\rho$-equivariance it is thus enough to define $\eta$ on the set of endpoints $\tilde e_0$ and $\tilde e_1$. Now define $\eta(\tilde e_i)=g_ix_i$, where $\Gamma(\tilde e_i)=g_iP$ and $x_i\in P$ is an element as provided by Lemma~\ref{lemma:Bruhat}. By construction, $\delta\eta$ is a Bruhat cocycle descending to a Bruhat cocycle on $\overline\T$. The freedom in the choice of $\eta$ is exactly the action by coboundaries preserving $\tau$. This proves the result.
\end{proof}

\subsubsection{Explicit formulas via Ptolemy coordinates}  Note that the labelings of the long edges are canonically determined, whereas a short edge is canonically determined if and only if it connects two counter diagonal long edges. The result below, generalizing the correspondence in Figure~\ref{fig:FundamentalCorrespondence}, gives explicit formulas for the canonically determined edges in terms of Ptolemy coordinates when at most one Ptolemy coordinate per face is zero (the edge-degenerate case; see Definition~\ref{def:PartitionType}). For $a\in\C$ and $b\in\C\setminus\{0\}$ let 
\begin{equation}
x(a)=\begin{pmatrix}1&a\\0&1\end{pmatrix},\qquad q(b)=\begin{pmatrix}0&-b^{-1}\\b&0\end{pmatrix}, \qquad d(b)=\begin{pmatrix}b&0\\0&b^{-1}\end{pmatrix}.
\end{equation}

\begin{proposition}\label{prop:ConcreteFormula} Let $\alpha_{ij}$ and $\beta^k_{ij}$ be the labelings of long and short edges of a Bruhat cocycle coming from a decoration $(g_0P,g_1P,g_2P,g_3P)$. If $c_{ij}$ are the Ptolemy coordinates, we have
\begin{equation}
\begin{aligned}
\alpha_{ij}&=q(c_{ij})\text{ if } c_{ij}\neq 0.\\
\beta_{ij}^k&=x(\frac{c_{ij}}{c_{ik}c_{kj}}) \text{ if }c_{ik}c_{kj}\neq 0.\\
\alpha_{ij}&=-d(c_{ik}^{-1}c_{kj}) \text{ if } c_{ij}=0\text{ and } c_{ik}c_{kj}\neq 0.
\end{aligned}
\end{equation}
\end{proposition}
\begin{proof}
The first formula only involves a single edge and is thus independent of whether or not the remaining Ptolemy coordinates are zero. The second only involves the face $ijk$ and holds by the fundamental correspondence if $c_{ij}\neq 0$ and for $c_{ij}=0$ as well by analytic continuation. The third is a consequence of the first two using the cocycle condition. 
\end{proof}

\begin{remark}\label{rm:CoboundaryAction}
Identifying $P$ with $\C$ via $\Matrix{1}{x}{0}{1}\leftrightarrow x$ allows us to view the labelings of the short edges near an ideal vertex $v$ of $M$ as complex vectors. The action of a coboundary taking a vertex $v$ to $x\in\C$ (keeping all other vertices fixed) then corresponds to moving $v$ by $x$ (see Figure~\ref{fig:CoboundaryAction}). 
\end{remark}

\begin{figure}[htb]
\hspace{0.5cm}\begin{minipage}[b]{0.3\textwidth}
\scalebox{0.75}{\input{figures_gen/CoboundaryAction.tex}}
\end{minipage}
\begin{minipage}[b]{0.55\textwidth}
\scalebox{0.75}{\input{figures_gen/NaturalCocycleDefn.tex}}
\end{minipage}\\
\begin{minipage}[t]{0.48\textwidth}
\caption{The action by coboundary corresponds to moving the ``center'' point.}\label{fig:CoboundaryAction}
\end{minipage}
\begin{minipage}[t]{0.48\textwidth}
\caption{The natural cocycle from a decoration.}\label{fig:NaturalCocycleDefn}
\end{minipage}
\end{figure}

\subsection{$\PSL(2,\C)$-representations}
The analogue of Theorem~\ref{thm:OneToOneGeneral} for $\PSL(2,\C)$-representations and $\PSL(2,\C)$ Bruhat cocycles also holds. The proof is identical to that of Theorem~\ref{thm:OneToOneGeneral} using the obvious analogues of Lemma~\ref{lemma:Bruhat} and Corollary~\ref{cor:LocalOneToOne} for $\PSL(2,\C)$.

\subsection{Non-boundary-unipotent representations} Let $M$ be as in Section~\ref{sec:Enhanced}. Recall that a decoration of $\rho\colon\pi_1(M)\to\SL(2,\C)$ is a $\rho$-equivariant assignment of $P$-cosets to the triangular faces of $\Ttildebar$. A \emph{fattened Bruhat cocycle} is a cocycle on the fattened decomposition restricting to a Bruhat cocycle on each truncated simplex and labeling each face pairing edge by diagonal elements.
\begin{theorem}  There is a one-to-one correspondence between decorated $\SL(2,\C)$-representations and fattened Bruhat cocycles.
\end{theorem}
\begin{proof}
The proof is similar to that of Theorem~\ref{thm:OneToOneGeneral} using the fattened decomposition instead of the regular one.
\end{proof}


\section{The edge relations}\label{sec:EdgeRelations}
We now define a relation among the Ptolemy coordinates of a decorated (boundary-unipotent) representation. This relation is a consequence of the Ptolemy relations in the case when all Ptolemy coordinates are nonzero, but if some Ptolemy coordinates are zero, this relation is independent of the Ptolemy relations. Hence, when defining the invariant Ptolemy variety this relation must be imposed in addition to the Ptolemy relations.

Let $K$ be the space obtained by cyclically gluing together ordered simplices $\Delta_0,\dots,\Delta_{N-1}$ along a common edge. We order the vertices of each simplex such that the common edge is the $01$ edge of each simplex, and such that the orientations induced by the orderings agree (see Figure~\ref{fig:EdgeNbh}).
We refer to the edges (other than the common edge) as \emph{top}, \emph{bottom}, and \emph{center} edges respectively.  

\begin{lemma}\label{lemma:edgeRelationEquivalence}
Let $c$ be a Ptolemy assignment on $K$ where all the Ptolemy coordinates of the top and bottom edges of $K$ are nonzero. We then have
\begin{equation}\label{eq:EdgeRelations}
\sum_{k=0}^{N-1} \frac{c_{23,k}}{c_{12,k}c_{13,k}}=0\iff\sum_{k=0}^{N-1}\frac{c_{23,k}}{c_{02,k}c_{03,k}}=0.
\end{equation}
Moreover, if the Ptolemy coordinate $c_{01}$ of the interior edge is also nonzero, both equations are satisfied. 
\end{lemma}
\begin{proof}
We first assume that $c_{01}\neq 0$. We have
\begin{equation}\label{eq:First}
\begin{aligned}
c_{01}\sum_{k=0}^{N-1} \frac{c_{23,k}}{c_{12,k}c_{13,k}}&=\sum_{k=0}^{N-1} \frac{c_{02,k}c_{13,k}-c_{03,k}c_{12,k}}{c_{12,k}c_{13,k}}\\
&=\sum_{k=0}^{N-1} \frac{c_{03,k-1}c_{13,k}-c_{03,k}c_{13,k-1}}{c_{13,k-1}c_{13,k}}\\
&=\sum_{k=0}^{N-1}\Big(\frac{c_{03,k-1}}{c_{13,k-1}}-\frac{c_{03,k}}{c_{13,k}}\big)\\
&=0.
\end{aligned}
\end{equation}
Here, the first equality follows from the Ptolemy relations~\eqref{eq:PtolemyRelation}, and the second from the identification relations $c_{12,k}=c_{13,k-1}$ and $c_{02,k}=c_{03,k-1}$ (indices modulo $N$). By a similar computation, one also has
\begin{equation}\label{eq:Second}
c_{01}\sum_{k=0}^{N-1} \frac{c_{23,k}}{c_{02,k}c_{03,k}}=0.
\end{equation}
Since we are assuming that $c_{01}\neq 0$, \eqref{eq:First} and \eqref{eq:Second} together prove the second statement.

If $c_{01}=0$, the Ptolemy relations become $c_{03,k}c_{12,k}=c_{02,k}c_{13,k}$, which together with the identification relations $c_{12,k}=c_{13,k-1}$ and $c_{02,k}=c_{03,k-1}$ imply that the ratio between $c_{12,k}c_{13,k}$ and $c_{02,k}c_{03,k}$ is independent of $k$. This concludes the proof.
\end{proof}

\begin{figure}[htb]
\hspace{0.5cm}\begin{minipage}[b]{0.55\textwidth}
\scalebox{0.8}{\input{figures_gen/EdgeNbh.tex}}
\end{minipage}
\hspace{-1cm}
\begin{minipage}[b]{0.33\textwidth}
\scalebox{0.75}{\input{figures_gen/EdgeRelationTop.tex}}
\end{minipage}\\
\begin{minipage}[t]{0.55\textwidth}
\caption{A Ptolemy assignment on $K$.}\label{fig:EdgeNbh}
\end{minipage}
\begin{minipage}[t]{0.44\textwidth}
\caption{Bruhat cocycle on $\overline K$ (near top vertex).}\label{fig:EdgeRelationTop}
\end{minipage}
\end{figure}

\begin{definition}\label{def:EdgeRelation}
We call the (equivalent) relations~\eqref{eq:EdgeRelations} \emph{edge relations} around the center edge of $K$.
\end{definition}

\begin{remark}
When all Ptolemy coordinates are nonzero, the fundamental correspondence provides a natural cocycle on the space $\overline K$ obtained from $K$ by truncating the simplices. In this case, the edge relations are equivalent to the relation coming from the fact that the product of short edges around the top and bottom is $I$ (see Figure~\ref{fig:EdgeRelationTop}). When some of the Ptolemy coordinates are zero this relation is no longer automatically satisfied and must be imposed.
\end{remark}

\subsection{Edge relations for $\PSL(2,\C)$ Ptolemy assignments}
The obvious analogue of Lemma~\ref{lemma:edgeRelationEquivalence} for $\PSL(2,\C)$ Ptolemy assignments still holds and we define the edge relations for $\PSL(2,\C)$ Ptolemy assignments by the exact same formula~\ref{eq:EdgeRelations}.

\subsection{Edge relations for enhanced Ptolemy assignments}
Let $c$ be an enhanced Ptolemy assignment on $K$ where all the Ptolemy coordinates of the top and and bottom edges of $K$ are nonzero.
The identification relations are given by $c_{12,k}=t_kh_kc_{13,k-1}$, $c_{02,k}=b_kh_kc_{03,k-1}$ and $c_{01,k}=c_{01,k-1}t_kb_k$, where the $t_k$, $b_k$ and $h_k$ are monomials in the $m_s$ and $l_s$ (see Figure~\ref{fig:FatEdgeNbh}). An elementary modification (we leave the details to the reader) of the proof of Lemma~\ref{lemma:edgeRelationEquivalence} shows that we have  
\begin{equation}\label{eq:EdgeEquationsEnhanced}
\sum_{k=0}^{N-1}\Big(\frac{c_{23,k}}{c_{12,k}c_{13,k}}\prod_{j=1}^k t_j^2\Big)=0\iff\sum_{k=0}^{N-1}\Big(\frac{c_{23,k}}{c_{02,k}c_{03,k}}\prod_{j=1}^k b_k^2\Big)=0,
\end{equation}
and that both are satisfied if the $c_{01,k}$ are nonzero. We refer to these as \emph{edge relations}. The geometric interpretation of the edge relations is given in Figure~\ref{fig:FatEdgeRelationTop}.

\begin{figure}[htb]
\hspace{0.5cm}\begin{minipage}[b]{0.55\textwidth}
\scalebox{1.2}{\input{figures_gen/FattenedEdgeNbh.tex}}
\end{minipage}
\hspace{-1cm}
\begin{minipage}[b]{0.35\textwidth}
\scalebox{0.75}{\input{figures_gen/FattenedEdgeRelationTop.tex}}
\end{minipage}\\
\begin{minipage}[t]{0.45\textwidth}
\caption{Fattened version of Figure~\ref{fig:EdgeNbh}.}\label{fig:FatEdgeNbh}
\end{minipage}
\hspace{-0.5cm}
\begin{minipage}[t]{0.5\textwidth}
\caption{Fattened Bruhat cocycle (top view).}\label{fig:FatEdgeRelationTop}
\end{minipage}
\end{figure}

\section{Transitive edge sets and degeneracy types}\label{sec:TransitiveEdgeSets}
Note that a decoration $D$ determines a distinguished subset of the edge set of $\T$, namely those whose Ptolemy coordinates are zero. We refer to this subset as the \emph{zero-set} of $D$.


\begin{definition}\label{def:TransitiveEdgeSubset}
A \emph{transitive edge set} is a subset $E$ of the edges of $\T$ such that if two edges on a face are in $E$, so is the third. The set of transitive edge sets is denoted by $\TE(\T)$. Given $E\in\TE(\T)$, we refer to the edges in $E$ as \emph{zero-edges}. To stress the dependence on $\T$ we sometimes denote a transitive edge set as $(\T,E)$ instead of $E$.
\end{definition}

\begin{lemma}\label{lemma:TransitiveDec} The zero-set of a decoration is transitive.
\end{lemma}
\begin{proof}
If $g_0P$, $g_1P$ and $g_2P$ are the cosets assigned to the vertices of a face, the Ptolemy coordinates are given by $c_{ij}=\det(v_i,v_j)$, where $v_i=g_i\left(\begin{smallmatrix}1\\0\end{smallmatrix}\right)$. Hence, a Ptolemy coordinate $c_{ij}$ is zero if and only if $v_i$ and $v_j$ are linearly dependent. The result now follows from the fact that linear dependence of vectors is transitive.
\end{proof}

\begin{definition} Let $E\in\TE(\T)$ be a transitive edge set. A face of $\T$ is \emph{degenerate} if all its three edges are zero-edges. A simplex of $\T$ is degenerate if all of its six edges are zero-edges.
\end{definition}

\begin{definition}\label{def:PartitionType}
We divide the transitive edge sets into the following types:
\begin{align*}
\textbf{Non-degenerate:}& \text{ No zero-edges (there is a unique such).}\\
\textbf{Edge-degenerate:}& \text{ Some zero-edges, but no degenerate faces.}\\
\textbf{Face-degenerate:}& \text{ Some degenerate faces, but no degenerate simplices.}\\
\textbf{Simplex-degenerate:}& \text{ Some, but not all, simplices are degenerate.}\\
\textbf{Totally degenerate:}& \text{ All simplices are degenerate (there is a unique such).}
\end{align*}
\end{definition}
Clearly, any $E\in\TE(\T)$ falls into exactly one of these types.
\subsection{Decorations, Ptolemy assignments and Bruhat cocycles}
\begin{definition} Let $E\in\TE(\T)$. A decoration is of \emph{type $E$} if its zero-edges agree with $E$. A Bruhat cocycle $\tau$ has \emph{type $E$} if the set of long edges that are diagonal agrees with the zero-edges of $E$. A Ptolemy assignment has type $E$ if the set of edges whose Ptolemy coordinate is zero agrees with the zero-edges of $E$.
\end{definition}
Note that the one-to-one-correspondence in Theorem~\ref{thm:OneToOneGeneral} preserves the type. We stress that the type depends on the triangulation.
\begin{remark}
Decorations of type $E$ may not exist. For example, if there is an edge loop where all but one edge is zero, the argument in the proof of Lemma~\ref{lemma:TransitiveDec} shows that $E$ cannot be the zero-set of a decoration.
\end{remark}

\subsection{The totally degenerate edge set}\label{sec:TotallyDegenerate}
\begin{proposition}\label{prop:TotallyDegenerate}
If a decoration of a representation $\rho$ is totally degenerate, then $\rho$ is reducible.
\end{proposition}
\begin{proof} If all Ptolemy coordinates are zero, the decoration must take all ideal vertices of $\widetilde M$ to the same $B$-coset. This is only possible if $\rho$ is reducible.
\end{proof}
\begin{proposition}\label{prop:ReducibleDegenerate} If a decoration of a reducible representation $\rho$ is not totally degenerate, then $\rho$ is boundary-degenerate. 
\end{proposition}
\begin{proof} A reducible representation has a totally degenerate decoration. The only way it can also have a decoration which is not totally degenerate is if a boundary component is a sphere or collapsed.
\end{proof}

\section{The Ptolemy variety of a transitive edge set}\label{sec:InvariantPtolemy}
We now define a Ptolemy variety $P(\T,E)$ for each transitive edge set $E$, which is not totally degenerate. We shall see that this variety parametrizes decorations of type $E$, and that the representation corresponding to an element in $P(\T,E)$ can be recovered explicitly via the corresponding Bruhat cocycle. 
\subsection{Edge-degenerate edge sets}

Let $E\in\TE(\T)$ be edge-degenerate. Then each face of $\T$ has at most one zero-edge. 
\begin{proposition}\label{prop:RecoverDec}
Let $c$ be a Ptolemy assignment on a simplex $\Delta$. If each face has at most one Ptolemy coordinate which is zero, then $c$ is the Ptolemy assignment of a unique decoration.
\end{proposition}
\begin{proof}
By reordering the vertices if necessary, we may assume that $c_{01}$, $c_{12}$ and $c_{23}$ are nonzero. Let
\begin{equation}
g_0=I,\quad g_1=q(c_{01}),\quad g_2=g_1x(\frac{c_{02}}{c_{01}c_{12}})q(c_{12}),\quad g_3=g_2x(\frac{c_{13}}{c_{12}c_{23}})q(c_{23}).
\end{equation}
The Ptolemy coordinates of the decoration $(g_0P,g_1P,g_2P,g_3P)$ then agree with $c$ as is shown by explicitly computing $\det\big(g_{i}\left(\begin{smallmatrix}1\\0\end{smallmatrix}\right),g_{j}\left(\begin{smallmatrix}1\\0\end{smallmatrix}\right)\big)$. Uniqueness follows from Corollary~\ref{cor:LocalOneToOne} since the natural cocycle is determined up to coboundaries by the Ptolemy coordinates (Proposition~\ref{prop:ConcreteFormula}).
\end{proof}

Let $C(E)$ be a union of disjoint cylindrical neighborhoods of the zero-edges of $E$. The space $M\setminus C(E)$ decomposes as a union of truncated simplices with the zero-edges being ``chopped'' (see Figure~\ref{fig:ChoppedCocycle}).  
\begin{corollary}\label{cor:ChoppedCocycle}
A Ptolemy assignment $c$ of type $E$ canonically determines a $G$-cocycle on the space $M\setminus C(E)$.
\end{corollary}
\begin{proof}
By Proposition~\ref{prop:RecoverDec} and Corollary~\ref{cor:LocalOneToOne}, $c$ determines up to coboundaries a Bruhat cocycle on each truncated simplex. By Proposition~\ref{prop:ConcreteFormula}, all the long edges are canonically determined by the Ptolemy coordinates, but a short edge near a zero-edge is only determined up to the coboundary action. However, as shown in Figure~\ref{fig:ChoppedCocycle}, each chopped truncated simplex inherits canonical edge labelings. The fact that the chopped cocycles match up follows from the identification relations.
\end{proof}

\begin{figure}[htb]
\hspace{0.5cm}\begin{minipage}[b]{0.55\textwidth}
\scalebox{0.75}{\input{figures_gen/RemoveCylinder.tex}}
\end{minipage}
\hspace{1cm}
\begin{minipage}[b]{0.3\textwidth}
\scalebox{0.85}{\input{figures_gen/RemoveCylinderTop.tex}}
\end{minipage}\\
\begin{minipage}[t]{0.5\textwidth}
\caption{Natural cocycle after removing a neigborhood of a zero-edge.}\label{fig:ChoppedCocycle}
\end{minipage}
\begin{minipage}[t]{0.46\textwidth}
\caption{The cocycle extends if and only if the edge relation holds.}\label{fig:RemoveCylinder}
\end{minipage}
\end{figure}

The link of each zero-edge is a complex $K$ as in Section~\ref{sec:EdgeRelations} (embedded in $\Mtildehat$) and since $E$ is only edge-degenerate, the top and bottom edges of $K$ are all nonzero. Hence, for a Ptolemy assignment of type $E$ the edge relations (Definition~\ref{def:EdgeRelation}) around each zero-edge are well defined. 

\begin{definition}\label{def:PtolemyMild} Let $E$ be edge-degenerate. The Ptolemy variety $P(\T,E)$ is the quasi-affine algebraic set defined by the usual relations (as in Definition~\ref{def:PtolemyVariety}) together with the edge relations~\eqref{eq:EdgeRelations} around zero-edges and the relations $c_e=0$ if $e$ is a zero-edge and $c_e\neq 0$ otherwise. 
\end{definition}
Note that each element in $P(\T,E)$ is a Ptolemy assignment of type $E$.

\begin{theorem}\label{thm:OneOneMild}
Let $E\in \TE(\T)$ be edge-degenerate. There is a one-to-one correspondence
\begin{equation}\label{eq:OneOneMild}
\left\{\txt{Decorated, boundary-unipotent\\representations of type $E$}\right\}\overset{\text{1:1}}{\longleftrightarrow}P(\T,E)\overset{\text{1:1}}{\longleftrightarrow}\left\{\txt{Bruhat cocycles of type $E$\\ up to coboundaries}\right\}
\end{equation}
\end{theorem}
\begin{proof}
By Theorem~\ref{thm:OneToOneGeneral} all we need to prove is that a Ptolemy assignment of type $E$ determines a Bruhat cocycle. By Corollary~\ref{cor:ChoppedCocycle}, we have a canonical cocycle on $M\setminus C(E)$. By the van Kampen theorem this extends to a Bruhat cocycle on $M$ if and only if the product of the labelings around each cylinder is $I$, which is a consequence of the edge relations. The freedom in the choice of extension is exactly the coboundary action (see e.g.~Remark~\ref{rm:CoboundaryAction}).
\end{proof}

\begin{definition}
For $\sigma\in H^2(\widehat M;\Z/2\Z)$ define $P^\sigma(\T,E)$ as in Definition~\ref{def:PtolemyMild}, but using the identification relations~\eqref{eq:EdgeIdentificationsSigma}.
\end{definition}

\begin{theorem}\label{thm:OneOneMildSigma}
Let $E\in \TE(\T)$ be edge-degenerate. There is a one-to-one correspondence
\begin{equation}\label{eq:OneOneMildSigma}
\cxymatrix{{
\left\{\txt{Type $E$ decorated\\boundary-unipotent\\$\pi_1(M)\to\PSL(2,\C)$}\right\}\ar@/_1pc/@{<->}[rr]_-{1:1}&P^\sigma(\T,E)\ar[l]_-{1:k}\ar[r]^-{k:1}&\left\{\txt{Bruhat $\PSL(2,\C)$-cocycles\\type $E$, obstruction class $\sigma$\\up to coboundaries}\right\},}}
\end{equation}
\end{theorem}
\begin{proof}
As in the proof of Theorem~\ref{thm:OneOneMild} an element in $P^\sigma(\T,E)$ determines a Bruhat cocycle up to coboundaries. The fact that this map is $k:1$ follows from the elementary fact that two Ptolemy assignments determine the same Bruhat cocycle if and only if they differ by an element in $Z^1(\widehat M;\Z/2\Z)$.
\end{proof}

\begin{definition}
For $M$ as in Section~\ref{sec:Enhanced} define the enhanced Ptolemy variety $\E P(\T,E)$ as in Definition~\ref{def:PtolemyMild} using the identification relations from Definition~\ref{def:EnhancedPtolemy} and the edge relations~\eqref{eq:EdgeEquationsEnhanced}.
\end{definition}

\begin{theorem}\label{thm:OneOneMildEnhanced}
Let $E\in \TE(\T)$ be edge-degenerate and $M$ as in Section~\ref{sec:Enhanced}. There is a one-to-one correspondence
\begin{equation}\label{eq:OneOneMildEnhanced}
\left\{\txt{Decorated representations\\of type $E$}\right\}\overset{\text{1:1}}{\longleftrightarrow}\E P(\T,E)\overset{\text{1:1}}{\longleftrightarrow}\left\{\txt{Fattened Bruhat cocycles of type $E$\\ up to coboundaries}\right\}
\end{equation}
\end{theorem}
\begin{proof}
The proof is identical to that of Theorem~\ref{thm:OneOneMild} using the fattened decomposition and the geometric interpretation of the edge relations~\eqref{eq:EdgeEquationsEnhanced} given in Figure~\ref{fig:FatEdgeRelationTop}.
\end{proof}

\subsection{Embeddings in affine space}
\begin{definition}\label{def:VirtualEdges}
A \emph{virtual edge} of $M$ is an element in the set
\begin{equation}
V(M) =\big(I(\widetilde M)\times I(\widetilde M)\big)\big/\pi_1(M)
\end{equation}
of pairs of ideal points in $\widetilde M$ up to the $\pi_1(M)$-action. 
\end{definition}
Given an ideal triangulation $\T$, each oriented edge $e$ of $\T$ determines a virtual edge, namely the orbit of $(e_0,e_1)$. 
Hence, $\T$ determines a (finite) subset $V(\T)$ of $V(M)$.

For each subset $V$ of $V(M)$ let $A(V)$ denote the affine space with coordinate ring equal to the polynomial ring generated by $V$.  For $V\subset W$ we have a canonical projection $\pi\colon A(W)\to A(V)$ onto a direct summand. 
We wish to embed the Ptolemy variety $P(\T,E)$ in $A(V(M))$. The idea is that a decoration assigns a Ptolemy coordinate to every virtual edge, not just the edges of $\T$.

\begin{proposition}\label{prop:Embeddings} Let $E\in\TE(\T)$ be non- or edge-degenerate. There is a canonical embedding of $P(\T,E)$ in the (infinite dimensional) affine space $A(V(M))$. Furthermore, for each finite subset $W$ of $V(M)$ containing $V(\T)$ the map
\begin{equation}
\xymatrix{P(\T,E)\ar@{^{(}->}[r]&A(V(M))\ar[r]^-\pi&A(W)}
\end{equation} 
is an embedding of $P(\T,E)$ in the finite dimensional affine space $A(W)$.
\end{proposition}
\begin{proof}
By definition, $P(\T,E)$ is a quasi-affine subset of $A(V(\T))$. To embed $P(\T,E)$ in $A(V(M))$, we must construct a surjective map from the coordinate ring $\mathcal O_{A(V(M))}$ of $A(V(M))$ to the coordinate ring $\mathcal O$ of $P(\T,E)$. 
By Corollary~\ref{cor:ChoppedCocycle} we have a canonical Bruhat cocycle on the space $M\setminus C(E)$ with edges labeled by matrices in $\SL(2,\mathcal O)$. This cocycle canonically lifts to a cocycle on $\widetilde{M\setminus C(E)}$. Given a virtual edge $e=(e_0,e_1)$, one obtains a matrix $B_e$ in $\SL(2,\mathcal O)$ by selecting a path in $\widetilde{M\setminus C(E)}$ joining $e_0$ to $e_1$. Up to left and right multiplication by elements in $P$ the matrix $B_e$ is independent of the choice of path, so the determinant $c_e=\det(\Vector{0}{1},B_e\Vector{0}{1})$ 
is well defined. Moreover, if $e$ corresponds to an edge of $\T$, $c_e$ is simply the Ptolemy coordinate of $e$. One thus obtains a surjection $\mathcal O_{A(V(M))}\to\mathcal O$ by taking a virtual edge $e$ to $c_e$.
The second statement is obvious.
\end{proof}

\begin{remark} The same result holds for $P^\sigma(\T,E)$. For $\E P(\T,E)$ one must consider embeddings into an affine space generated by pairs of triangular faces of $\widetilde{\overline\T}$ instead of virtual edges. 
\end{remark}

\subsection{Face-degenerate edge sets}
\begin{definition}\label{def:Descendant}
Let $\T'$ be a triangulation whose edge set contains that of $\T$. A \emph{descendant} of $E\in\TE(\T)$ is an element $E'\in\TE(\T')$ such that $E'$ agrees with $E$ on the edges of $\T$. If $E'\in\TE(\T')$ is a descendent of $E\in\TE(\T)$ we write $E'_{|\T}=E$.
\end{definition}
\begin{lemma} Let $E\in \TE(\T)$ be face-degenerate and suppose that there are $k$ degenerate faces. If $\T'$ is the triangulation obtained from $\T$ by performing a 2-3 move at each degenerate face then $E$ has $2^k$ descendants each of which is edge-degenerate.
\end{lemma}
\begin{proof}
Performing a 2-3 move adds an edge which can be either zero or nonzero without violating transitivity. In either case, none of the new faces are degenerate (see Figure~\ref{fig:ModerateDescendant}). This concludes the proof.
\end{proof}

\begin{figure}[htb]
\hspace{0.8cm}
\begin{minipage}[b]{0.45\textwidth}
\scalebox{0.8}{\input{figures_gen/ModerateDescendant.tex}}
\end{minipage}
\hfill
\begin{minipage}[b]{0.45\textwidth}
\scalebox{0.8}{\input{figures_gen/WildDescendant.tex}}
\end{minipage}\\
\hspace{-0.5cm}
\begin{minipage}[t]{0.49\textwidth}
\caption{A 2-3 move, face-degenerate edge set.}\label{fig:ModerateDescendant}
\end{minipage}
\begin{minipage}[t]{0.49\textwidth}
\caption{A 2-3 move, simplex-degenerate edge set.}\label{fig:WildDescendant}
\end{minipage}
\end{figure}

\begin{definition}\label{def:PtolemyModerate} Let $E\in \TE(\T)$ be face-degenerate and let $\T'$ be the triangulation obtained from $\T$ by performing 2-3 moves between all degenerate faces. Define
\begin{equation}\label{eq:PtolemyVarietyModerate}
P(\T,E)=\bigcup_{E'_{|\T}=E}P(\T',E'),
\end{equation}
where the (clearly disjoint) union is taken inside $A(V(M))$ or the finite-dimensional $A(W)$ where $W$ is the union of all $V(\T')$. We define $P^\sigma(\T,E)$ and $\E P(\T,E)$ similiarly.
\end{definition}



\begin{theorem}\label{thm:OneOneModerate}
The one-to-one correspondences~\eqref{eq:OneOneMild}, \eqref{eq:OneOneMildSigma}, and \eqref{eq:OneOneMildEnhanced} still hold if $E$ is face-degenerate. 
\end{theorem}
\begin{proof}
If $\T'$ is a decoration with edge set containing $\T$, the set of decorations of type $E\in\TE(\T)$ is the disjoint union, over the descendants $E'$ of $E$, of the sets of decorations of type $E'$. The result now follows from Theorem~\ref{thm:OneOneMild}.
\end{proof}

\subsection{Simplex-degenerate edge sets}
Given a simplex-degenerate edge set $E\in \TE(\T)$. Let $d(E)$ denote the number of degenerate simplices.
\begin{lemma} Let $E\in \TE(\T)$ be simplex-degenerate and let $\T'$ be a triangulation obtained from $\T$ be performing a 2-3 move at a face between a degenerate and a non-degenerate simplex. Then $E$ has a unique descendant $E'\in \TE(\T)$. Furthermore $d(E')=d(E)-1$.
\end{lemma}
\begin{proof}
Transitivity holds if and only if the new edge is non-zero. Hence, none of the new simplices are degenerate (see Figure~\ref{fig:WildDescendant}). This proves the result.
\end{proof}
\begin{corollary}\label{cor:RemoveDegeneracy} Let $E\in \TE(\T)$ be simplex-degenerate. There exists a triangulation $\T'$ obtained from $\T$ by performing 2-3 moves such that $E$ has a unique descendant $E'\in \TE(\T')$, which is face-degenerate.
\end{corollary}
\begin{proof}
Repeatedly perform 2-3 moves at a face between a degenerate and a non-degenerate simplex until there are no more degenerate simplices.
\end{proof}

\begin{definition}\label{def:PtolemyWild} Let $E\in \TE(\T)$ be simplex-degenerate and let $E'$ and $\T'$ be as in Corollary~\ref{cor:RemoveDegeneracy}. Define
\begin{equation}
P(\T,E)=P(\T',E')
\end{equation}
and define $P^\sigma(\T,E)$ and $\E P(\T,E)$ similarly.
\end{definition}

Note that as a subset of $A(V(M))$ the definition is independent of the choice of $\T'$.
%

\begin{theorem}\label{thm:OneOneWild}
The one-to-one correspondences~\eqref{eq:OneOneMild}, \eqref{eq:OneOneMildSigma}, and \eqref{eq:OneOneMildEnhanced} hold for all transitive edges sets $E$ that are not totally degenerate.
\end{theorem}
\begin{proof}
This follows from Theorem~\ref{thm:OneOneModerate} and Corollary~\ref{cor:RemoveDegeneracy}.
\end{proof}

\section{The invariant Ptolemy varieties}\label{sec:Invariant}
For the non-degenerate edge set $E$, define $P(\T,E)$ to be the standard Ptolemy variety (Definition~\ref{def:PtolemyVariety}). The Ptolemy variety for the other types were defined in Definitions~\ref{def:PtolemyMild}, \ref{def:PtolemyModerate}, and \ref{def:PtolemyWild}. We denote the totally degenerate edge set where all edges are zero-edges by $0\in\TE(\T)$. 
\begin{definition}\label{def:InvariantPtolemy} The \emph{invariant Ptolemy variety} $\overline P(\T)$ is defined by
\begin{equation}
\overline P(\T)=\bigcup_{E\in\TE(\T)\setminus\{0\}} P(\T,E),
\end{equation}
where the (disjoint) union is taken inside $A(V(M))$. We similarly define $\overline P^\sigma(\T)$ and $\E\overline P(\T)$.
\end{definition}

\subsection{Affine coverings}
A priori, the invariant Ptolemy varieties are only unions of quasi-affine subsets of $A(V(M))$.
\begin{proposition}\label{prop:ProofAffine}
The invariant Ptolemy varieties are indeed varieties. Changing the triangulation changes the varieties by biregular isomorphisms.
\end{proposition}





\begin{proof}
We prove this for $\overline P(\T)$; the proof for the other variants are identical. We must prove that $\overline P(\T)$ has an open affine cover. 
For each transitive edge set $(\T,E)$ consider the triangulation $T(E)$ defined as follows: if $E$ is non-degenerate or edge-degenerate, $T(E)=\T$; if $E$ is moderately, or simplex-degenerate, define $T(E)$ to be a triangulation $\T'$ as in Definitions~\ref{def:PtolemyModerate} or~\ref{def:PtolemyWild}, respectively. By definition, the invariant Ptolemy variety $\overline P(\T)$ is then given by
\begin{equation}
\overline P(\T)=\bigcup_{\substack{E\in \TE(\T)\setminus\{0\}\\F\in\TE(T(E)),\,F_{|\T}=E}}\!\!\!\!P(T(E),F).
\end{equation}
Note that all transitive edge sets $(T(E),F)$ in the above union are (at worst) edge-degenerate. Let $W\in V(M)$ be the union of all $V(T(E))$, i.e.~$W$ is the set of virtual edges $e$ such that $e$ is an edge of one of the $T(E)$. By Proposition~\ref{prop:Embeddings}, we may regard $\overline\P(\T)$ as a subset of the finite dimensional affine space $A(W)$. Recall that $A(W)$ is given by a coordinate $c_e$ for each virtual edge in $W$. 
For a transitive edge set $(\T',E')$ let
\begin{equation}
U_{(\T',E')}=\big\{\prod_{e\in E_{\neq 0}'}c_e\neq0\big\}\in A(V(M)),
\end{equation} 
where $E_{\neq 0}'$ denotes the set of nonzero edges of $(\T',E')$. Clearly, $U_{(\T',E')}$ is Zariski open, and
\begin{equation}
(\T',E')\subset(\T',E'')\iff U_{(\T',E')}\subset U_{(\T',E'')}.
\end{equation}
Also, for any transitive edge sets $(T',E')$ and $(\T'',E'')$ we clearly have
\begin{equation}
E'_{=0}\cap E''_{\neq 0}\neq \emptyset  \implies P(\T',E')\cap U_{(\T'',E'')}=\emptyset.
\end{equation}
From this it follows that for any $E\in\TE(\T)$ and any $F\in\TE(T(E))$, we have 
\begin{equation}\label{eq:OpenIntersect}
\overline P(\T)\cap U_{(T(E),F)} = \big(\bigcup_{(\T,E')\leq(\T,E)} P(\T,E')\big)\cap U_{(T(E),F)}\subset A(W).
\end{equation}
By Proposition~\ref{prop:Embeddings} the right-hand side of \eqref{eq:OpenIntersect} canonically embeds in $A(V(T(E)))$. Let $X$ be the affine variety in $A(V(T(E)))$ cut out by the Ptolemy relations, identification relations, and edge orientation relations for $T(E)$ together with the edge relations around zero-edges of $F$. Since $(T(E),F))$ is (at worst) edge-degenerate, there are no points in $X\cap U_{(T(E),F)}$ that violate transitivity, so all points in $X\cap U_{(T(E),F)}$ are in $P(\T,E')$ for some $E'\in\TE(\T)$ with $(\T,E')<(\T,E)$.  
It thus follows from~\eqref{eq:OpenIntersect} that $X\cap U_{(T(E),F)}$ equals $\overline P(\T)\cap U_{(T(E),F)}$.
Since $U_{(T(E),F)}$ is the non-vanishing set of a single polynomial, $X\cap U_{(T(E),F)}$ is affine.
This concludes the proof of existence of an affine cover of $\overline P(\T)$.

Given another triangulation $\T'$, $\overline P(\T)$ and $\overline P(\T')$ are clearly isomorphic as sets since both parametrize non totally degenerate decorations. The fact that they are also equal as varieties follows from the fact that the images of $\overline P(\T)$ and $\overline P(\T')$ in $A(V(M))$ are equal (see the proof of Proposition~\ref{prop:Embeddings}).  
\end{proof}



\section{The reduced Ptolemy varieties}
The diagonal action on decorations, Ptolemy assignments, and natural cocycles extends canonically to the case where some of the Ptolemy coordinates are zero and some of the long edges are diagonal instead of counterdiagonal. We define the reduced Ptolemy variety $\overline P_{\red}(\T)$ to be the quotient of $\overline P(\T)$ by the diagonal action. 
In~\cite{PtolemyField} we showed that the reduced Ptolemy variety $P(\T)_{\red}$ can be computed by setting appropriately chosen Ptolemy coordinates equal to $1$: 
\begin{theorem}[{\cite[Prop.~1.16]{PtolemyField}}]\label{thm:SetEdgesToOne} Let $G$ be a connected graph in the one-skeleton of $\widehat M$, which has fundamental group $\Z$ and contains all vertices. The reduced Ptolemy variety $P(\T)_{\red}$ is isomorphic to the subvariety of $P(\T)$ where the Ptolemy coordinates of all edges in $G$ are $1$. The same result holds for $P^\sigma(\T)_{\red}$.
\end{theorem}
\begin{remark}
In~\cite{PtolemyField} this is only proved when $G$ is a so-called ``maximal tree with $1$- or $3$-cycle'', but the proof can be trivially extended to any graph with fundamental group $\Z$ containing all vertices.
\end{remark}

\begin{theorem}\label{thm:Reduced} Let $E\in \TE(\T)$ be edge-degenerate. Then, there exists a graph $G$ of nonzero edges which has fundamental group $\Z$ and contains all vertices. The reduced Ptolemy variety $P(\T,E)_{\red}$ is isomorphic to the subvariety of $P(\T,E)$ where the Ptolemy coordinates of all edges in $G$ are $1$. The same result holds for $P^\sigma(\T,E)_{\red}$ and $\E P(\T,E)_{\red}$.
\end{theorem}
\begin{proof}
Since $E$ is edge-degenerate, the set of nonzero edges connects all vertices of $\T$ and contains cycles. Thus, we can choose a graph $G$ with the above properties. Hence, the result follows from Theorem~\ref{thm:SetEdgesToOne}. The proofs trivially extend to the case of $\E P(\T,E)_{\red}$.
\end{proof}

\section{Summary of the proofs of main results}

Theorem~\ref{thm:MainResult} is an immediate consequence of Theorem~\ref{thm:OneOneWild} and Proposition~\ref{prop:ProofAffine}. Theorem~\ref{thm:OneToOneOverNonDegenerate} follows from Corollary~\ref{cor:Freedom}. Theorem~\ref{thm:MainThmPSL} follows from Theorem~\ref{thm:OneOneWild} and the analogue of Corollary~\ref{cor:Freedom} for $\PSL(2,\C)$. Theorem~\ref{thm:MainThmEnhanced} follows from Theorem~\ref{thm:OneOneWild} and Corollary~\ref{cor:Freedom}.

\section{Examples}\label{sec:Examples}
Let $M$ be the manifold \texttt{m009} from the SnapPy census~\cite{SnapPy}. The census triangulation $\T$ of $M$ is shown in Figure~\ref{fig:Trigm009}.
\begin{figure}[htb]
\scalebox{0.66}{\input{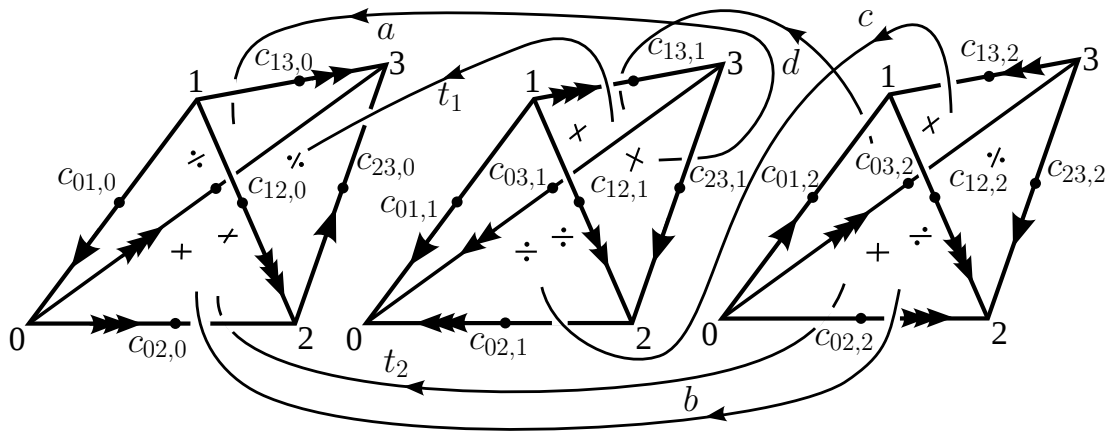}}
\caption{Census triangulation of \texttt{m009}. The signs on the faces indicate the obstruction cocycle $\sigma^3$.}\label{fig:Trigm009}
\end{figure}
We refer to the three edges as \emph{edge 1}, \emph{edge 2}, and \emph{edge 3}, according to the number of arrow heads. By inspecting the figure, one checks that there are four transitive edge sets: The non-degenerate edge set where all edges are nonzero, the edge-degenerate edge set where only edge 1 is zero, the edge-degenerate edge set where only edge 2 is zero, and the totally degenerate edge set (which we ignore, see Section~\ref{sec:TotallyDegenerate}). The edge relation around edge 1 is given by
\begin{equation}\label{eq:EdgeOneRelation}
\frac{c_{23,0}}{c_{21,0}c_{13,0}}+\frac{c_{01,1}}{c_{03,1}c_{31,1}}+\frac{c_{10,0}}{c_{12,0}c_{20,0}}+\frac{c_{01,2}}{c_{03,2}c_{31,2}}+\frac{c_{23,1}}{c_{21,1}c_{13,1}}+\frac{c_{32,2}}{c_{30,2}c_{02,2}}=0
\end{equation}
and the edge relation around edge 2 is given by
\begin{equation}\label{eq:EdgeTwoRelation}
\frac{c_{02,0}}{c_{01,0}c_{12,0}}+\frac{c_{30,1}}{c_{31,1}c_{10,1}}+\frac{c_{20,2}}{c_{23,2}c_{30,2}}+\frac{c_{12,1}}{c_{13,1}c_{32,1}}=0.
\end{equation}
Since there is no transitive edge set where edge 3 is zero (except the totally degenerate edge set), Theorem~\ref{thm:Reduced} implies that the reduced Ptolemy variety (all variants) is given by setting the Ptolemy coordinate of edge 3 equal to $1$.

\subsection{The Ptolemy variety $\overline P(\T)$}\label{sec:SLexample}
The identification relations corresponding to the three edges are
\begin{equation}
\begin{gathered}
c_{23,0}=-c_{23,2}=-c_{01,1}=c_{01,2}=-c_{01,0}=-c_{23,1}\\
c_{13,0}=c_{12,1}=-c_{13,2}=-c_{03,1}=c_{13,0}\\
c_{13,1}=c_{03,2}=c_{02,0}=c_{12,2}=-c_{02,1}=c_{03,0}=c_{02,2}=c_{12,0}.
\end{gathered}
\end{equation}
Letting $x=c_{23,0}$, $y=c_{13,0}$ and $z=c_{13,1}$, the Ptolemy relations $c_{03,i}c_{12,i}+c_{01,i}c_{23,i}=c_{02,i}c_{13,i}$ become
\begin{equation}\label{eq:m009PtolemySL}
z^2-x^2=zy,\qquad -y^2+x^2=-z^2,\qquad z^2-x^2=-zy.
\end{equation}
One easily checks that the only solution to this is $x=y=z=0$, so the Ptolemy variety $\overline P(\T)$ is empty. There are thus no irreducible, boundary-unipotent $\SL(2,\C)$-representations of \texttt{m009} (as mentioned in Section~\ref{sec:PSLPtolemy}, the geometric $\PSL(2,\C)$-representation has no boundary-unipotent lift to $\SL(2,\C)$).

\subsection{The Ptolemy varieties $\overline P^\sigma(\T)$}\label{sec:PSLexample}

An elementary cohomology computation shows that $H^2(\widehat M;\Z/2\Z)=\Z/2\Z\oplus\Z/2\Z$ and that the three non-trivial classes are represented by cocycles $\sigma^i\in C^2(\widehat M;\Z/2\Z)$ whose restrictions $\sigma^i_k$ to the $k$th simplex of $\T$ are given by
\begin{equation}
\begin{aligned}
\sigma^1_0&=f_{0}^*+f_{2}^*,\quad&\sigma^1_1&=f_{0}^*+f_{1}^*,\quad&\sigma^1_2&=0,\\
\sigma^2_0&=f_{0}^*+f_{1}^*,\quad&\sigma^2_1&=f_{0}^*+f_{3}^*,\quad&\sigma^2_2&=f_{0}^*+f_{1}^*,\\
\sigma^3_0&=f_{1}^*+f_{2}^*,\quad&\sigma^3_1&=f_{1}^*+f_{3}^*,\quad&\sigma^3_2&=f_{0}^*+f_{1}^*,
\end{aligned}
\end{equation}
where $f_{i}^*\in C^2(\Delta;\Z/2\Z)$ denotes the cochain taking the face $f_{i}$ opposite vertex $i$ to $-1$ and all other faces to $1$. The cocycle $\sigma^3$ is indicated in Figure~\ref{fig:Trigm009}. 

One easily checks that $\sigma^i_k=\delta(\eta_k^i)$, where $\eta^i_k$ is given by
\begin{equation}
\begin{aligned}
\eta^1_0&=\varepsilon_{13}^*,\quad&\eta^1_1&=\varepsilon_{23}^*,\quad&\eta^1_2&=0,\\
\eta^2_0&=\varepsilon_{23}^*,\quad&\eta^2_1&=\varepsilon_{12}^*,\quad&\eta^2_2&=\varepsilon_{23}^*,\\
\eta^3_0&=\varepsilon_{03}^*,\quad&\eta^3_1&=\varepsilon_{02}^*,\quad&\eta^3_2&=\varepsilon_{23}^*,
\end{aligned}
\end{equation}
where $\varepsilon_{ij}^*\in C^1(\Delta;\Z/2\Z)$ is the cochain taking $\varepsilon_{ij}$ to $-1$ and all other edges to $1$.
The Ptolemy variety for the trivial obstruction class is equal to $\overline P(\T)$, which is trivial, as we saw earlier.
Recall that the reduced Ptolemy variety is obtained by setting the Ptolemy coordinate $z$ of edge 3 equal to 1.
\subsubsection{Ptolemy variety for $\sigma^1$}
The identification relations~\eqref{eq:EdgeIdentificationsSigma} are
\begin{equation}
\begin{gathered}
c_{23,0}=-c_{23,2}=-c_{01,1}=c_{01,2}=-c_{01,0}=c_{23,1}\\
c_{13,0}=-c_{12,1}=c_{13,2}=c_{03,1}=c_{13,0}\\
c_{13,1}=c_{03,2}=c_{02,0}=c_{12,2}=-c_{02,1}=c_{03,0}=c_{02,2}=c_{12,0}.
\end{gathered}
\end{equation}
Again, letting $x=c_{23,0}$, $y=c_{13,0}$ and $z=c_{13,1}$, the Ptolemy relations become
\begin{equation}\label{eq:m009PtolemyPSL1}
z^2-x^2=zy,\qquad -y^2-x^2=-z^2,\qquad z^2-x^2=zy.
\end{equation}
Setting $z=1$, the equations then have the three solutions
\begin{equation}
(x,y,z)=(0,1,1),\qquad (x,y,z)=(-1,0,1),\qquad (x,y,z)=(1,0,1).
\end{equation}
However, not all of these solutions are valid, since the edge equations may not be satisfied. The edge relation~\eqref{eq:EdgeOneRelation} around edge 1 (defined when $y$ and $z$ are nonzero) is
\begin{equation}
\frac{x}{-zy}+\frac{-x}{y(-z)}+\frac{x}{z(-z)}+\frac{x}{z(-y)}+\frac{x}{yz}+\frac{x}{-zz}=\frac{-2x}{z^2}=0
\end{equation}
which is satisfied when $x$ is zero. Similarly, the edge relation~\eqref{eq:EdgeTwoRelation} around edge 2 (defined when $z$ and $x$ are nonzero) is
\begin{equation}
\frac{z}{-xz}+\frac{-y}{-zx}+\frac{-z}{-x(-z)}+\frac{-y}{z(-x)}=\frac{2y}{zx}-\frac{2}{x}=0,
\end{equation}
which is not satisfied. Hence, the reduced Ptolemy variety for $\sigma^1$ consists of a single point given by $(x,y,z)=(0,1,1)$.

\subsubsection{Ptolemy variety for $\sigma^2$}
The Ptolemy relations are
\begin{equation}
z^2+x^2=yz,\qquad y^2+x^2=-z^2,\qquad z^2+x^2=-yz,
\end{equation}
which have no non-trivial solution. Hence, the Ptolemy variety $\overline P^{\sigma^2}(\T)$ is empty.
\subsubsection{Ptolemy variety for $\sigma^3$}
The Ptolemy relations are
\begin{equation}
z^2+x^2=yz,\qquad y^2+x^2=-z^2,\qquad z^2+x^2=-yz,
\end{equation} 
and setting $z=1$ they are equivalent to
\begin{equation}
 x^2 + y + 1=0,\qquad y^2 + y + 2=0,\qquad z=1
\end{equation}
Hence, there are no solutions with $x$ or $y$ being $0$, and $\overline P^{\sigma^3}(\T)_{\red}$ is defined over the number field $\Q(w)$, where $w^4+w^2+2=0$, and is given by $x=w$, $y=-w^2-1$, $z=1$.

\subsection{The Ptolemy variety $\E\overline P(\T)$ and the $A$-polynomial}\label{sec:Apolyexample}
A fundamental rectangle for the boundary of $M$ is shown in Figure~\ref{fig:FundamentalRectangle}. Using the rules illustrated in Figure~\ref{fig:FundRectIdentifications} we obtain that the identification relations are given by

\begin{figure}[htb]
\scalebox{0.7}{\input{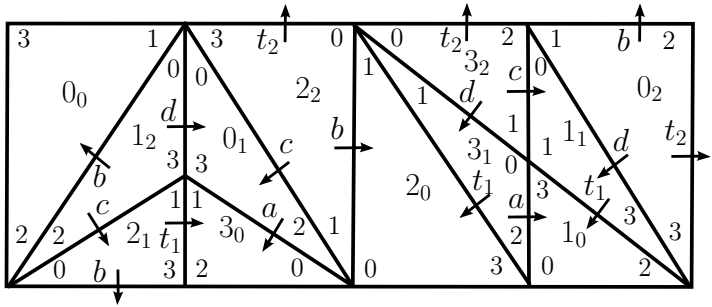}}
\caption{Fundamental rectangle for \texttt{m009}.}\label{fig:FundamentalRectangle}
\end{figure}

\begin{equation}
\begin{gathered}
c_{23,0}\overset{t_2}{=}-m^2c_{23,2}\overset{c}{=}-m^2c_{01,1}\overset{d}{=}m^2c_{01,2}\overset{b}{=}-mc_{01,0}\overset{a}{=}-c_{23,1},\\
c_{13,0}\overset{t_1}{=}c_{12,1}\overset{c}{=}-c_{13,2}\overset{d}{=}-c_{03,1},\\
c_{13,1}\overset{d}{=}c_{03,2}\overset{t_2}{=}m^{-1}lc_{02,0}\overset{b}{=}m^{-1}lc_{12,2}\overset{c}{=}-m^{-1}lc_{02,1}\overset{a}{=}lc_{03,0}\overset{t_2}{=}mc_{02,2}\overset{b}{=}c_{12,0},
\end{gathered}
\end{equation}
where the symbol $\overset{\alpha}{=}$ indicates that the identification is via the face pairing $\alpha$. Hence, the Ptolemy relations become
\begin{equation}\label{eq:ExmlPtolemy}
mz^2-lx^2-m^2yz,\qquad m^2ly^2-lx^2-m^3z^2,\qquad m^5z^2-lx^2+m^3lyz.
\end{equation}
One easily checks that there are no non-trivial solutions with $x=0$. The edge relation~\eqref{eq:EdgeEquationsEnhanced} around edge 2 is 
\begin{equation}
\frac{ml^{-1}z}{-m^{-1}xz}+\frac{y}{(-z)m^{-2}x}+\frac{-m^{-1}z}{-m^{-2}x(-z)}+\frac{y}{zx}=0\iff m^2z+m^2ly+mlz-ly=0
\end{equation}
Adding this equation to the Ptolemy relations~\eqref{eq:ExmlPtolemy} and substituting $z=1$, one can check (using magma~\cite{Magma}) that the system is equivalent to
\begin{equation}
\begin{gathered}
        x^2 + yl - m^8 + 3m^6 + m^5l + m^4 - m^3l - 3m^2 - ml=0,\\
        y^2l + yl^2 + m^4l - m^3 - m^2l - ml^2 - m - l=0,\\
        ym + yl + m^4 - m^2 - ml - 1=0,\\
        yl^3 - yl - m^5l + m^4l^2 + m^3l + m^2 - ml^3 + 2ml - l^2=0,\\
        m^6l - 2m^4l - m^3l^2 - m^3 - 2m^2l + l=0.
\end{gathered}
\end{equation}
This shows that the $A$-polynomial of \texttt{m009} is given by
\begin{equation}
A(m,l)=m^6l - 2m^4l - m^3l^2 - m^3 - 2m^2l + l.
\end{equation}
It also follows that $\E\overline P(\T)$ is a branched cover over the $A$-polynomial curve of degree 2. Another magma computation shows that it is given explicitly by
\begin{equation}\label{eq:EPOverA}
x^2=-\frac{-m^4-2m^3l+ml}{l^2(m^2-1)},\qquad y=-\frac{m^2+ml}{m^2l - l},\qquad z=1,\qquad A(m,l)=0.
\end{equation}
Note that $y$ and $x^2$ are regular functions on the $A$-polymial curve.
\subsection{Recovering the representations}\label{sec:RecoverRepExample}
The dual triangulation of \texttt{m009} has an oriented edge for each of the face pairings $a$, $b$, $c$, $d$, $t_1$, and $t_2$. The edges $t_1$ and $t_2$ form a maximal tree, so the fundamental group of \texttt{m009} is generated by $a$, $b$, $c$, and $d$. By inspecting Figure~\ref{fig:FundamentalRectangle} we see that
\begin{equation}
\pi_1(M)=\langle a,b,c,d\bigm \vert cd^{-1}a^{-1},cb^{-1}d^{-1}ba,ca^{-1}bd^{-1}\rangle,
\end{equation}
and that the meridian $\mu$ and longitude $\lambda$ are given by
\begin{equation}
\mu=acb^{-1},\qquad \lambda=d^{-1}cd^{-1}bc^{-1}db^{-1}.
\end{equation}
One also checks that the generators may be represented by edge paths in the truncated complex as follows:
\begin{equation}
\begin{aligned}
a&=\beta^0_{23,0}\alpha_{03,0}\beta^0_{23,1}\alpha_{03,1}\beta^3_{01,1}(\beta^2_{01,0})^{-1}\alpha_{02,0}^{-1}\\
b&=\alpha_{02,0}\beta^2_{01,0}\alpha_{12,0}^{-1}(\beta^0_{23,0})^{-1}\\
c&=\beta^0_{23,0}\alpha_{03,0}\beta^3_{01,0}\alpha_{12,1}^{-1}(\beta^3_{01,2})^{-1}\alpha_{03,2}^{-1}\\
d&=\alpha_{02,0}\beta^2_{01,0}\alpha_{12,0}^{-1}.
\end{aligned}
\end{equation}
One can then compute the representations explicitly using the formulas for $\alpha_{ij}$ and $\beta^k_{ij}$ given by Proposition~\ref{prop:ConcreteFormula}.

\subsubsection{The representation with obstruction class $\sigma^1$}\label{sub:Reducible}
We obtain
\begin{equation}\label{eq:ReducedRep}
a=c=\begin{pmatrix}0&-1\\1&0\end{pmatrix},\qquad b=d=I,\qquad\mu=-\lambda=I.
\end{equation}
This shows that the Ptolemy variety can detect reducible representations. It is boundary-degenerate as it must be by Proposition~\ref{prop:ReducibleDegenerate}.

\subsubsection{The representations with obstruction class $\sigma^3$}
We obtain
\begin{equation}
\begin{gathered}
a=\begin{pmatrix}w^3+w&1\\1&-w\end{pmatrix},\qquad b=\begin{pmatrix}1&-w\\-w&w^2+1\end{pmatrix},\qquad c=\begin{pmatrix}w^3&1\\w^2+1&-w\end{pmatrix}\\
d=\begin{pmatrix}1&0\\-w&1\end{pmatrix},\qquad \mu=\begin{pmatrix}1&w\\0&1\end{pmatrix},\qquad \lambda=\begin{pmatrix}-1&2w^3+w\\0&-1\end{pmatrix}.
\end{gathered}
\end{equation}
One easily checks that $H^1(\widehat M;\Z/2\Z)=\Z/2\Z$. Hence, by Theorem~\ref{thm:MainThmPSL} there should be only two $\PSL(2,\C)$ representations, not four. Indeed, replacing $w$ by its Galois conjugate $-w$ corresponds to conjugating the representation by the diagonal matrix with entries $\sqrt{-1}$ and $-\sqrt{-1}$. Note that the fixed field of the Galois isomorphism $w\mapsto -w$ is $\Q(\sqrt{-7})$, which is the shape field of \texttt{m009}. The two representations detected are the geometric representation and its complex conjugate.

\subsubsection{The non-boundary-unipotent representations}\label{sub:TautologicalRep}
Representing the generators by edge paths in the fattened truncated complex, we have
\begin{equation}
\begin{aligned}
a&=\beta^0_{23,0}\alpha_{03,0}\beta^0_{23,1}\alpha_{03,1}\beta^3_{01,1}(\beta^2_{01,0})^{-1}\alpha_{02,0}^{-1}\\
b&=\alpha_{02,0}\beta^2_{01,0}\alpha_{12,0}^{-1}M^{-1}L^{-1}(\beta^0_{23,0})^{-1}\\
c&=\beta^0_{23,0}\alpha_{03,0}\beta^3_{01,0}\alpha_{12,1}^{-1}(\beta^3_{01,2})^{-1}\alpha_{03,2}^{-1}L^{-1}\\
d&=\alpha_{02,0}\beta^2_{01,0}\alpha_{12,0}^{-1}L^{-1},
\end{aligned}
\end{equation}
where $M=\Matrix{m}{0}{0}{m^{-1}}$ and $L=\Matrix{l}{0}{0}{l^{-1}}$. Using this, we obtain
\begin{equation}
\begin{gathered}
a=\begin{pmatrix}\frac{-(m+l)x}{m^4-m^2}&\frac{l^2}{m^2}\\\frac{-1}{ml}&\frac{-lx}{m^2}\end{pmatrix},\qquad b=\begin{pmatrix}\frac{1}{m^2}&\frac{l^2x}{m^3}\\\frac{-x}{m^2l}&\frac{m^2+ml}{m^2-1}\end{pmatrix},\qquad c=\begin{pmatrix}\frac{-(ml+1)x}{m^4-m^2}&\frac{l^2}{m}\\\frac{m+l}{l^2(m^2-1)}&\frac{-lx}{m}\end{pmatrix}\\
d=\begin{pmatrix}m^{-1}&0\\\frac{-x}{ml}&m\end{pmatrix},\qquad \mu=\begin{pmatrix}m&\frac{-l^2x}{m^2}\\0&m^{-1}\end{pmatrix},\qquad \lambda=\begin{pmatrix}l&\frac{(-l^3+l)x}{m^3-m}\\0&l^{-1}\end{pmatrix},
\end{gathered}
\end{equation}
where $x$ is given by~\eqref{eq:EPOverA}.

\begin{remark}\label{rm:EnhancedPSLExample}
As mentioned in Section~\ref{sec:PtolemyBrief} one also has a Ptolemy variety $\E\overline P(\T)^\sigma_{\red}$. This variant only depends on the image of $\sigma$ in the cokernel of the canonical map $H^1(\partial M;\Z/2\Z)\to H^2(M,\partial M;\Z/2\Z)$. For \texttt{m009} this cokernel is $\Z/2\Z$, and $\E\overline P(\T)^\sigma_{\red}$ for the generator turns out to have two additional components giving rise to curves of $\PSL(2,\C)$-representations that don't lift to $\SL(2,\C)$ (and therefore not detected by $\E\overline P(\T)_{\red}$). One of these components is a component of dihedral representations that are deformations of the representation~\eqref{eq:ReducedRep}. 
\end{remark}

\subsection*{Acknowledgment}
We thank Stavros Garoufalidis, Walter Neumann and Henry Segerman for helpful discussions.

\bibliographystyle{plain}
\bibliography{BibFile}

\end{document}